\documentclass[11pt, twoside]{article}
\pdfoutput=1

\usepackage[margin=1in]{geometry}

\newcommand{\innOm}[2]{(#1,#2)_\Omega}
\newcommand{\innGa}[2]{\langle #1,#2\rangle_\Gamma}
\newcommand{\duaOm}[2]{[#1,#2]_\Omega}
\newcommand{\duaGa}[2]{[#1,#2]_\Gamma}

\newcommand{\norm}[1]{\left\|#1\right\|}

\usepackage{graphicx}
\usepackage{amsmath,amsfonts,amssymb}
\usepackage[mathscr]{eucal}
\usepackage{bm}

\usepackage{amsthm}
\usepackage{cite}
\usepackage{hyperref}

\usepackage[nameinlink]{cleveref}

\numberwithin{equation}{section}

\newtheorem{theorem}{Theorem}[section]
\newtheorem{lemma}[theorem]{Lemma}
\newtheorem{prop}[theorem]{Proposition}
\newtheorem{corollary}[theorem]{Corollary}
\newtheorem{thm}[theorem]{Theorem}

\theoremstyle{definition}
\newtheorem{definition}[theorem]{Definition}

\theoremstyle{remark}
\newtheorem{remark}[theorem]{Remark}

\usepackage{fancyhdr}
\begin{document}

\title{A New HDG Method for Dirichlet Boundary Control of Convection Diffusion PDEs I: High Regularity}

\author{Weiwei Hu%
	\thanks{Department of Mathematics, Oklahoma State
		University, Stillwater, OK (weiwei.hu@okstate.edu). W.~Hu was supported in part by a postdoctoral fellowship for the annual program on Control Theory and its Applications at the Institute for Mathematics and its Applications (IMA) at the University of Minnesota.}%
	\and
	Mariano Mateos%
	\thanks{Dpto. de Matem\'aticas. Universidad de Oviedo, Campus de Gij\'on, Spain (mmateos@uniovi.es).  M.\ Mateos was supported by the Spanish Ministerio de Econom\'{\i}a y Competitividad under project MTM2014-57531-P.}%
	\and
	John~R.~Singler%
	\thanks{Department of Mathematics
		and Statistics, Missouri University of Science and Technology,
		Rolla, MO (\mbox{singlerj@mst.edu}, ywzfg4@mst.edu). J.~Singler and Y.~Zhang were supported in part by National Science Foundation grant DMS-1217122.  J.~Singler and Y.~Zhang thank the IMA for funding research visits, during which some of this work was completed.}
	\and
	Yangwen Zhang%
	\footnotemark[3]
}

\date{}

\maketitle

\begin{abstract}
	We propose a new hybridizable discontinuous Galerkin (HDG) method to approximate the solution of a Dirichlet boundary control problem governed by an elliptic convection diffusion PDE.  Even without a convection term, Dirichlet boundary control problems are well-known to be very challenging theoretically and numerically.  Although there are many works in the literature on Dirichlet boundary control problems for the Poisson equation, the authors are not aware of any existing theoretical or numerical analysis works for convection diffusion Dirichlet control problems.  We make two contributions.  First, we obtain well-posedness and regularity results for the Dirichlet control problem.  Second, under certain assumptions on the domain and the target state, we obtain optimal a priori error estimates in 2D for the control for the new HDG method.  As far as the authors are aware, there are no existing comparable results in the literature.  We present numerical experiments to demonstrate the performance of the HDG method.
\end{abstract}

\section{Introduction}
We consider the following  Dirichlet boundary control problem. Let $\Omega\subset \mathbb{R}^{d} $ $ (d\geq 2)$ be a Lipschitz polyhedral domain  with boundary $\Gamma = \partial \Omega$.  The goal is to find the optimal control $ u \in L^2(\Gamma) $ that minimizes the cost function
\begin{align}
J(u)=\frac{1}{2}\| y- y_{d}\|^2_{L^{2}(\Omega)}+\frac{\gamma}{2}\|u\|^2_{L^{2}(\Gamma)}, \quad \gamma>0, \label{cost1}
\end{align}
subject to the elliptic convection diffusion equation
\begin{equation}\label{Ori_problem}
\begin{split}
-\Delta y+\bm \beta\cdot\nabla y&=f \quad\text{in}~\Omega,\\
y&=u\quad\text{on}~\partial\Omega,
\end{split}
\end{equation}
where $ f \in L^2(\Omega) $ and the vector field $\bm{\beta}$ satisfies
\begin{align}\label{beta_con}
\nabla\cdot\bm{\beta} \le 0.
\end{align}
We make other smoothness assumptions on $ \bm \beta $ for our analysis.

Optimal control problems governed by convection diffusion equations play an important role in many scientific and engineering problems \cite{MR1766429}. Efficient and accurate numerical  methods are essential to successful applications of such optimal control problems.  There exist many contributions \cite{MR2302057,MR2595051,MR2486088,MR2178571,MR2068903,MR2550371} to numerical methods and algorithms for this kind of problem.  Despite this large amount of existing work on numerical methods for convection diffusion optimal control problems and also Dirichlet boundary control problems for the Poisson equation and other PDEs \cite{MR2272157,MR2558321,MR3070527,MR2347691,MR3317816,MR3614013,MR2020866,MR2567245,MR1632548,ravindran2017finite,MR3641789,MR2806572,ApelMateosPfeffererRosch17,MR1135991,MR1145711,MR1874072}, we are not aware of any existing work on the analysis and approximation of solutions for the above convection diffusion Dirichlet boundary control problem.  Work on this problem is an important step towards the analysis and approximation of Dirichlet boundary control problems for the Navier-Stokes equations and other fluids models.

In recent years, the discontinuous Galerkin (DG) methods have been proved very useful in solving a large range of computational fluids problems \cite{MR1241479,MR1860450,MR2031230,MR1929628} and optimal control problems for convection diffusion PDEs \cite{MR3022208,MR3416418,MR2773301,MR3307573,MR2587414,MR2644299,MR3149415}. However, one disadvantage of DG methods is the high number of degree of freedom compared to standard finite element methods.

Hybridizable Discontinuous Galerkin (HDG) methods, proposed by Cockburn et al.\ in \cite{MR2485455}, have the same advantages as typical DG methods but have many less globally coupled unknowns.  HDG methods are currently undergoing rapid development and have been used in many applications; see, e.g., \cite{MR2772094,MR2513831,MR2558780,MR2796169, MR3626531,MR3522968,MR3463051,MR3452794,MR3343926}.

Formally, the optimal control $ u \in L^2(\Gamma) $ and the optimal state $ y \in L^2(\Omega) $ minimizing the cost functional satisfy the optimality system
\begin{subequations}\label{eq_adeq}
	\begin{align}
	-\Delta y+\bm \beta\cdot\nabla y &=f\qquad\quad\text{in}~\Omega,\label{eq_adeq_a}\\
	y &= u\qquad\quad\text{on}~\partial\Omega,\label{eq_adeq_b}\\
	-\Delta z-\nabla\cdot(\bm{\beta} z) &=y-y_d\quad  \text{in}~\Omega,\label{eq_adeq_c}\\
	z&=0\qquad\quad~\text{on}~\partial\Omega,\label{eq_adeq_d}\\
	\nabla z \cdot\bm n - \gamma  u&=0\qquad\quad \  \text{on}~\partial\Omega.\label{eq_adeq_e}
	\end{align}
\end{subequations}
Even in the absence of the convection term, theoretical analysis, numerical discretization, and numerical analysis for the above optimal Dirichlet control problem and optimality system can be very challenging.  Difficulties arise due to the control entering in the Dirichlet boundary condition \eqref{eq_adeq_b} and also due to the normal derivative of the dual state on the boundary in \eqref{eq_adeq_e}.  See the references above on the Dirichlet boundary control problems for the Poisson equation for more information.

In order to apply an HDG method, a mixed formulation of the state equation \eqref{eq_adeq_a}--\eqref{eq_adeq_b} and of the adjoint state equation \eqref{eq_adeq_c}--\eqref{eq_adeq_d} must be used. The meaning of the state equation \eqref{eq_adeq_a} for Dirichlet boundary data in $L^2(\Gamma)$ must be made clear for this kind of formulation.  In \Cref{sec:Analysis_of_the_Dirichlet_Control_Problem}, we perform this analysis for the case of a 2D polygonal domain and also establish well-posedness and regularity results for the optimality system \eqref{eq_adeq}.

In a recent work \cite{HuShenSinglerZhangZheng_HDG_Dirichlet_control1}, we  approximate the solution of the Dirichlet boundary control problem for the Poisson equation using an existing HDG method.  This method uses polynomials of degree $k+1$ to approximate the state $y$ and dual state $z$ and  polynomials of degree $k \ge 0$ for the fluxes $\bm q = -\nabla  y $ and $ \bm p = -\nabla z$, respectively. Moreover, we also used polynomials of degree $k$ to approximate the numerical trace of the state and dual state on the edges (or faces) of the spatial mesh, which are the only globally coupled unknowns.  We obtained a superlinear convergence rate for the optimal control under certain basic assumptions on the desired state $ y_d $ and the domain $ \Omega $.  Despite the large amount of existing work on this problem, a similar convergence rate has only very recently been proved for one other numerical method: a finite element method on a special class of meshes \cite{ApelMateosPfeffererRosch17}.

However, it is not clear to the authors if existing HDG methods can guarantee the superlinear convergence rate as we obtained in \cite{HuShenSinglerZhangZheng_HDG_Dirichlet_control1} for elliptic convection diffusion PDEs.  Therefore, we devise a new HDG method in \Cref{sec:HDG} using polynomials of degree $k+1$ to approximate the state $y$, dual state $z$, and the numerical traces.  Moreover, we use polynomials of degree $k \ge 0$ for the fluxes $\bm q = -\nabla  y $ and $ \bm p = -\nabla z$, respectively.  In \Cref{sec:analysis}, we prove the same superlinear rate of convergence as in \cite{HuShenSinglerZhangZheng_HDG_Dirichlet_control1} for the control in 2D under certain assumptions on the largest angle of the convex polygonal domain and the smoothness of the desired state $ y_d $.  To give a specific example, for a rectangular 2D domain and $ y_d\in H^{1-\varepsilon}(\Omega)$, we obtain the following a priori error bounds for the state $ y $, adjoint state $ z $, their fluxes $ \bm{q} = -\nabla y $ and $ \bm{p} = -\nabla z $, and the optimal control $ u $:
\begin{align*}
&\norm{y-{y}_h}_{0,\Omega}=O( h^{3/2-\varepsilon} ),\quad  \;\norm{z-{z}_h}_{0,\Omega}=O( h^{3/2-\varepsilon} ),\\
&\norm{\bm{q}-\bm{q}_h}_{0,\Omega} = O( h^{1-\varepsilon} ),\quad  \quad \norm{\bm{p}-\bm{p}_h}_{0,\Omega} = O( h^{3/2-\varepsilon} ),
\end{align*}
and
\begin{align*}
&\norm{u-{u}_h}_{0,\Gamma} = O( h^{3/2-\varepsilon} ),
\end{align*}
for any $\varepsilon >0$.  The rate of convergence for the control is optimal.  We present numerical results in \Cref{sec:numerics} that precisely match the convergence theory for the control in 2D.

In the second part of this work \cite{HuMateosSinglerZhangZhang17}, we remove the assumptions required here on the convex polygonal domain and the desired state $ y_d $ and prove optimal convergence rates for the control.  Removing these assumptions on the domain and the desired state lower the regularity of the solution of the optimality system; therefore, regular HDG error analysis techniques are not applicable.  We perform a nonstandard HDG error analysis based on techniques from \cite{MR3508837,LiXie16_SINUM} to establish the low regularity convergence results.

We emphasize that this new HDG method may be of primary interest for boundary control problems such as the one considered here.  Existing HDG methods use order $ k $ polynomials for the numerical traces, which are the only globally coupled unknowns.  This new HDG method uses order $ k+1 $ polynomials for the numerical traces, and therefore it has a higher computational cost compared to existing HDG methods.  However, adding one polynomial degree to the space for the numerical traces is the only way we have found to guarantee the optimal convergence rate for the control.  The authors are not aware of any other application where this new HDG method will lead to an improved convergence analysis over existing HDG methods.

\section{Analysis of the Dirichlet Control Problem}
\label{sec:Analysis_of_the_Dirichlet_Control_Problem}

To begin, we set notation and prove some fundamental results concerning the optimality system for the control problem in the 2D case.

Throughout the paper we adopt the standard notation $W^{m,p}(\Omega)$ for Sobolev spaces on $\Omega$ with norm $\|\cdot\|_{m,p,\Omega}$ and seminorm $|\cdot|_{m,p,\Omega}$ . We denote $W^{m,2}(\Omega)$ by $H^{m}(\Omega)$ with norm $\|\cdot\|_{m,\Omega}$ and seminorm $|\cdot|_{m,\Omega}$. Specifically, $H_0^1(\Omega)=\{v\in H^1(\Omega):v=0 \;\mbox{on}\; \partial \Omega\}$.  We denote the $L^2$-inner products on $L^2(\Omega)$ and $L^2(\Gamma)$ by
\begin{align*}
(v,w)_{\Omega} &= \int_{\Omega} vw  \quad \forall v,w\in L^2(\Omega),\\
\left\langle v,w\right\rangle_{\Gamma} &= \int_{\Gamma} vw  \quad\forall v,w\in L^2(\Gamma).
\end{align*}
Define the space $H(\text{div},\Omega)$ as
\begin{align*}
H(\text{div},\Omega) = \{\bm{v}\in [L^2(\Omega)]^d, \nabla\cdot \bm{v}\in L^2(\Omega)\}.
\end{align*}
Duality between $H^{1}(\Omega)^*$ and $H^1(\Omega)$ will be denoted
$\duaOm{q}{r}$, while duality between $H^{-\varepsilon}(\Gamma)$ and $H^{\varepsilon}(\Gamma)$ for $0\leq \varepsilon\leq 1/2$ will be denoted $\duaGa{u}{v}$.

Throughout this section, we consider $\Omega$ a polygonal domain, not necessarily convex, and denote $\omega$ its biggest interior angle. Notice that $1/2<\pi/\omega<1$ for nonconvex domains and $1<\pi/\omega\leq 3$ for convex domains.  Furthermore, we assume in this section $ \bm \beta $ satisfies the following conditions:
\begin{equation}\label{eqn:beta_assumptions1}
\begin{split}
\bm \beta \in [L^\infty(\Omega)]^d,  \quad  \nabla\cdot\bm{\beta}\in L^\infty(\Omega),  \quad  \nabla \cdot \bm \beta \leq 0  &\quad  \mbox{for any $ \Omega $, and also}\\
\nabla \nabla \cdot \bm{\beta}\in[L^2(\Omega)]^d  &\quad  \mbox{if $ \Omega $ is convex.}
\end{split}
\end{equation}
Moreover, we assume the forcing $ f $ is identically zero in this section.  If this is not the case, then a simple change of variable as in \cite[pg.\ 3623]{AMPR-2015} can be used to eliminate the forcing.

\subsection{Study of the state equation}
Notice that for data $u\in L^2(\Gamma)$, we cannot expect to have a variational solution of the state equation \eqref{Ori_problem}. Therefore, we need a suitable concept of solution that makes the control-to-state operator continuous and that coincides with the variational solution for regular data.
Moreover, since HDG is based on a mixed formulation, it is also important to see how this concept of very weak solution extends to mixed formulations.

To define the concept of a very weak solution, we first introduce the adjoint problem and recall its regularity properties.
\begin{lemma}\label{ML1}For every $g\in L^2(\Omega)$ there exists a unique $z_g\in H^1_0(\Omega)\cap H^{t}(\Omega)$ for all $t\leq 2$ with $t<1+\pi/\omega$ such that
	\begin{equation}\label{ME01}
	-\Delta z_g -\nabla\cdot (\bm{\beta} z_g) = g\mbox{ in }\Omega,\ z_g = 0\mbox{ on }\Gamma.
	\end{equation}
	Moreover, $\partial_{\bm n} z_g\in H^{s}(\Gamma)$ for all $s\leq 1/2$ such that $s<\pi/\omega-1/2.$
	
	If, further, $g\in H^{t^*}(\Omega)$ for some $0\leq t^*<1$, then $z_g\in H^1_0(\Omega)\cap H^{t}(\Omega)$ for all $t\leq 2+t^*$ with $t<\min\{3,1+\pi/\omega\}$ and
	$\partial_{\bm n} z_g\in H^{s}(\Gamma)$ for all $s\leq 1/2+t^*$ such that $s<\min\{3/2,\pi/\omega-1/2\}.$
\end{lemma}
\begin{proof}
	Existence and uniqueness of the solution is standard. The regularity of $\bm{\beta}$ implies that $\nabla\cdot (\bm{\beta} z_g)\in L^2(\Omega)$, and hence $z_g\in H^t(\Omega)$ for all $t\leq 2$ such that $t<1+\pi/\omega$.
	For the regularity of the normal derivative in the case $s<1/2$,  apply \cite[Corollary 2.3]{AMPR-2015} and trace theory in \cite{Grisvard85}. For $s=1/2$, apply \cite[Lemma (A2)]{Casas-Mateos-Raymond-2009}.
	
	For the extra regularity result, we use that $z_g\in H^t(\Omega)$ for all $t\leq 2$ such that $t<1+\pi/\omega$, $\nabla\cdot\bm{\beta}\in L^\infty(\Omega)$, and $\nabla \nabla \cdot \bm{\beta}\in[L^2(\Omega)]^d$ to obtain that $\nabla\cdot (\bm{\beta} z_g)\in H^{t^*}(\Omega)$. Now we have that $-\Delta z_g\in H^{t^*}(\Omega)$ and standard regularity results in \cite{Grisvard85} lead to $z_g \in H^1_0(\Omega)\cap H^{t}(\Omega)$ for all $t\leq 2+t^*$ with $t<\min\{3,1+\pi/\omega\}$. The normal trace then satisfies that $\partial_n z_g\in \Pi_{i=1}^mH^s(\Gamma_i)$ for all $s\leq 1/2+t^*$ such that $s<\min\{3/2,\pi/\omega-1/2\}$, where $\Gamma_i$ denotes side $i$ of the boundary of $\Omega$. If $\pi/\omega<1$, then $s<1/2$ and
	$\Pi_{i=1}^mH^s(\Gamma_i)=H^s(\Gamma)$. If $\pi/\omega>1$, we use that $z_g=0$ on $\Gamma$ as in \cite[Section 4]{Casas-Gunther-Mateos2011} to prove that $\partial_n z_g=0$ on the corners of the domain. This implies that $\partial_n z_g$ is continuous  and hence it belongs to $H^s(\Gamma)$. 
\end{proof}

\begin{definition}Let $\varepsilon$ be a real number such that $0\leq \varepsilon\leq 1/2$ and $\varepsilon<\pi/\omega-1/2$.
	For     $u\in H^{-\varepsilon}(\Gamma)$, we say that $y\in L^2(\Omega)$ is a very weak solution of
	\begin{equation}\label{ME02}
	-\Delta y +\bm{\beta} \cdot \nabla y = 0\mbox{ in }\Omega,\ y=u\mbox{ on }\Gamma
	\end{equation}
	if and only if
	\begin{equation}\label{M03}
	\innOm{y}{g} + \duaGa{u}{\partial_n z_g} = 0,
	\end{equation}
	for all $g \in L^2(\Omega)$, where $z_g$ is the unique solution of \eqref{ME01}.
\end{definition}
\begin{remark}\label{MR23}The definition is meaningful thanks to the regularity of the normal derivative of $z_g$ provided in \Cref{ML3}. Eventually the case $\varepsilon=1/2$ must be discarded, but we keep it while we can cope with it. For our problem we need only to consider the case $\varepsilon =0$. We include the other cases for the sake of completeness and because the definition may be useful for problems with control or state constraints; see e.g., \cite[Section 6.2]{Mateos-Neitzel2015}
\end{remark}

\begin{lemma}\label{ML3}Let $s$ be a real number such that $-1/2\leq s <3/2$ and $s>1/2-\pi/\omega$.
	For every $u\in H^{s}(\Gamma)$, there exists a unique {\em very weak} solution $y\in H^{1/2+s}(\Omega)$ of \eqref{ME02} and
	\[\|y\|_{H^{1/2+s}(\Omega)}\leq C \|u\|_{H^{s}(\Gamma)}.\]
\end{lemma}
\begin{proof}
	The case  $s = -1/2$ can only happen in a convex domain and we can use the classic transposition method. The proof for $-1/2<s<0$ is as the the proof of Lemma 2.5 in \cite{AMPR-2015}.
	
	For $1/2\leq s<3/2$ we have that \eqref{ME02} has a unique variational solution $y$ and that it belongs to $H^{s+1/2}(\Omega)$; see \cite[Proof of Corollary 4.2]{AMPR-2015}. Integration by parts shows that $y$ is also a very weak solution.
	
	For $0 \leq s<1/2$ the result follows from interpolation. 
\end{proof}

Next we do the same for the mixed formulation.  From now on, we assume the polygonal domain $ \Omega $ is convex so that $ 1 < \pi/\omega \leq 3 $. First we state an existence and regularity result for the mixed formulation of the convection diffusion equation with regular data.
\begin{lemma}\label{ML25}For every $g\in L^2(\Omega)$, there exists a unique pair $(z_g,\bm p_g)\in H^1_0(\Omega)\times H(\textup{div},\Omega)$ such that
	\begin{subequations}
		\begin{align}
		\innOm{\bm p_g}{\bm r}-\innOm{z_g}{\nabla \cdot \bm r} &= 0, \label{ME1a}\\
		\innOm{\nabla\cdot(\bm p_g-\bm{\beta} z_g)}{w}&= \innOm{g}{w},\label{ME1b}
		\end{align}
	\end{subequations}
	for all $(\bm r, w)\in H(\textup{div},\Omega)\times L^2(\Omega)$. Moreover, $(z_g,\bm p_g)\in (H^2(\Omega)\cap H^1_0(\Omega))\times [H^1(\Omega)]^d$, $\partial_{\bm n}z_g=\bm p_g\cdot\bm n\in H^{1/2}(\Gamma)$ and
	\begin{equation}-\Delta z_g-\nabla\cdot(\bm{\beta} z_g) = g\mbox{ in }\Omega, z_g=0\mbox{ on }\Gamma\mbox{ and }\bm p_g = -\nabla z_g.\label{ME2}
	\end{equation}
	If, further, $g\in H^{t^*}(\Omega)$ for some $0\leq t^*<1$, then $(z_g,\bm p_g)\in (H^{t}(\Omega)\cap H^1_0(\Omega))\times [H^{t-1}(\Omega)]^d$ for all $t\leq 2+t^*$ with $t<\min\{3,1+\pi/\omega\}$ and
	$\partial_{\bm n} z_g=\bm p_g\cdot\bm n\in H^{s}(\Gamma)$ for all $s\leq 1/2+t^*$ such that $s<\min\{3/2,\pi/\omega-1/2\}.$
\end{lemma}
Notice that the notation $z_g$ is not contradictory. If $z_g$ is the solution of $\eqref{ME01}$, then $(z_g,-\nabla z_g)$ is the solution of \eqref{ME1a}--\eqref{ME1b}. Also, if
$(z_g,\bm p_g)$ is the solution of \eqref{ME1a}--\eqref{ME1b}, then $z_g$ is the solution of $\eqref{ME01}$. Therefore this lemma is a straightforward consequence of \Cref{ML1}

Now we must drop the case $\varepsilon=1/2$.
\begin{definition}\label{MD26}Let  $\varepsilon$ be a real number such that
	\[0\leq \varepsilon<1/2.\]
	For $u\in H^{-\varepsilon}(\Gamma)$, we say that $(y,\bm q)\in L^2(\Omega)\times [H^1(\Omega)^*]^d$ is a very weak solution
	of
	\begin{equation}-\Delta y +\bm{\beta}\cdot\nabla y= 0\mbox{ in }\Omega,\ y= u\mbox{ in }\Gamma,\ \bm q=-\nabla y\label{ME3}\end{equation}
	if and only if
	\begin{subequations}
		\begin{align}
		\duaOm{\bm q}{\bm r} -\innOm{y}{\nabla\cdot \bm r} + \duaGa{u}{\bm r\cdot \bm n} &= 0,\label{ME4a}\\
		\duaOm{\bm q+\bm{\beta} y}{\bm p_g} - \innOm{y \nabla\cdot\bm{\beta}}{z_g}  &= 0,\label{ME4b}
		\end{align}
	\end{subequations}
	for all $(\bm r, g)\in [H^1(\Omega)]^d\times L^2(\Omega) $, and $(z_g,\bm p_g)\in (H^2(\Omega)\cap H^1_0(\Omega))\times [H^1(\Omega)]^d$ is the unique solution of \eqref{ME1a}--\eqref{ME1b}.
\end{definition}
\begin{remark} For $\varepsilon=1/2$, the expression $\duaGa{u}{\bm r\cdot \bm n}$ is  meaningless even for $\bm r\in [C^\infty(\Omega)]^d$, since $\bm n$ has jump derivatives, and hence $\bm r\cdot\bm n\not\in H^{1/2}(\Gamma)$.
\end{remark}

\begin{theorem}\label{MT28}For every $u\in H^{-\varepsilon}(\Gamma)$, there exists a unique {\em very weak} solution $(y,\bm q)\in L^2(\Omega)\times [H^1(\Omega)^*]^d$ of \eqref{ME3}.
	Moreover, $(y,\bm q)\in H^{1/2-\varepsilon}(\Omega)\times [H^{1/2+\varepsilon}(\Omega)^*]^d$ and $y$ is the very weak solution of \eqref{ME02}.
\end{theorem}
\begin{proof}
	Let us first prove uniqueness of solution in the space $L^2(\Omega)\times [H^1(\Omega)^*]^d$. Take $u=0$ and let $(y,\bm q)\in L^2(\Omega)\times [H^1(\Omega)^*]^d$ be functions satisfying:
	\begin{subequations}
		\begin{align}
		\duaOm{\bm q}{\bm r}  &=\innOm{y}{\nabla\cdot \bm r},\label{ME5a}\\
		\duaOm{\bm q+\bm{\beta} y}{\bm p_g}  -  \innOm{y \nabla\cdot\bm{\beta}}{z_g}&=0,\label{ME5b}
		\end{align}
	\end{subequations}
	for all $(\bm r, g)\in [H^1(\Omega)]^d\times L^2(\Omega) $. Consider $g=y$. From equation \eqref{ME1b} and taking into account that $\bm p_y=-\nabla z_y$, we have that
	\begin{equation}\label{ME9}
	\innOm{y}{\nabla\cdot\bm p_y}+\innOm{\bm{\beta} y}{\bm p_y} - \innOm{y\nabla\cdot\bm{\beta}}{ z_y}=\innOm{y}{y}.
	\end{equation}
	Take $\bm r=\bm p_y$ in \eqref{ME5a}. We obtain
	\[
	\duaOm{\bm q}{\bm p_y} = \innOm{y}{\nabla\cdot \bm p_y}.
	\]
	Substitute this in \eqref{ME5b}
	\begin{equation}\label{ME10}
	\innOm{y}{\nabla\cdot\bm p_y}+\innOm{\bm{\beta} y}{\bm p_y} - \innOm{y\nabla\cdot\bm{\beta}}{ z_y}=0.
	\end{equation}
	From \eqref{ME9} and \eqref{ME10}, it is clear that $y=0$. From \eqref{ME5a} we have that $\bm q=0$ and uniqueness is proved.
	
	\medskip
	
	Existence is as follows. Take $y\in L^2(\Omega)$ the unique {\em very weak} solution of \eqref{ME02} and define $\bm q\in [H^1(\Omega)^*]^d$ by
	\[\duaOm{ q_i}{ r} = \innOm{y}{\partial_{x_i} r}-\duaGa{u}{r n_i},\]
	for all $ r\in H^1(\Omega)$, and $n_i$ is the $i-$th component of the vector $\bm n$. Again, this is well defined because we have made sure that $\varepsilon<1/2$,
	and the functions in $H^{\varepsilon}(\Gamma)$ now can have jump discontinuities.
\end{proof}
\begin{corollary}If $u\in H^{1/2+t^*}(\Gamma)$ for some $0\leq t^*<1$, then $(y,\bm q)\in H^{1+t^*}(\Omega)\times ( [H^{t^*}(\Omega)]^d \cap H(\mathrm{div},\Omega) ) $ and
	\begin{subequations}
		\begin{align}
		\innOm{\bm q}{\bm r} -\innOm{y}{\nabla\cdot \bm r} + \innGa{u}{\bm r\cdot \bm n} &= 0,\label{ME11a}\\
		\innOm{\nabla\cdot(\bm q +\bm{\beta} y)}{w} -\innOm{y\nabla\cdot\bm{\beta}}{w}   &= 0,\label{ME11b}
		\end{align}
		for all $(\bm r,w)\in H(\textup{div},\Omega)\times L^2(\Omega)$.
	\end{subequations}

\end{corollary}
For the sake of completeness, we say that $(y,\bm q)$ is a very weak solution of the mixed formulation \eqref{ME11a}--\eqref{ME11b} if it is a very weak solution of
\eqref{ME3} in the sense of \Cref{MD26}.
\subsection{Study of the control problem}
Now we are ready to study the control problem. Let us first formulate it using the mixed formulation.
\[\textup{(P)}\min_{u\in L^2(\Gamma)} J(u) = \frac12\|y_u-y_d\|^2_{L^2(\Omega)}+\frac{\gamma}{2}\|u\|^2_{L^2(\Gamma)},\]
where $(y_u,\bm q_u)\in L^2(\Omega)\times [H^1(\Omega)^*]^d$ is the very weak solution of
\begin{subequations}
	\begin{align}
	\innOm{\bm q_u}{\bm r} -\innOm{y_u}{\nabla\cdot \bm r} + \innGa{u}{\bm r\cdot \bm n} &= 0,\label{ME12a}\\
	\innOm{\nabla\cdot(\bm q_u +\bm{\beta} y_u)}{w} -\innOm{y_u\nabla\cdot\bm{\beta}}{w}  &= 0,\label{ME12b}
	\end{align}
	for all $(\bm r,w)\in H(\textup{div},\Omega)\times L^2(\Omega)$.
\end{subequations}
\begin{theorem}\label{MT210}
	Assume $\Omega$ is convex.  If $y_d\in H^{t^*}(\Omega)$ for some $0\leq t^*<1$, then problem \textrm{(P)} has a unique solution $ \bar u\in L^2(\Gamma)$.  Moreover, for any $ s \geq 1/2 $ satisfying $s\leq \frac 12 +t^*$ and $s<\min\{\frac 3 2,\frac{\pi}{\omega}-\frac 1 2\}$, we have $ \bar u\in H^s (\Gamma)$,
	\begin{align*}
		(\bar{\bm p},\bar y,\bar z) &\in [H^{s+\frac 12}(\Omega)]^d\times H^{s+\frac 12}(\Omega)\times (H^{s+\frac 32}(\Omega)\cap H_0^1(\Omega)),\\
		\bar {\bm q} &\in  [H^{s-\frac 12}(\Omega)]^d \cap H(\mathrm{div},\Omega),
	\end{align*}
	and $\partial_{\bm n} \bar z = \bar {\bm p} \cdot\bm n\in H^{s}(\Gamma)$ such that
	\begin{subequations}
		\begin{align}
		\innOm{\bar {\bm q}}{\bm r} -\innOm{\bar y}{\nabla\cdot \bm r} + \innGa{\bar u}{\bm r\cdot \bm n} &= 0,\label{mixed_a}\\
		\innOm{\nabla\cdot(\bar {\bm q}+\bm{\beta} \bar y)}{w} -\innOm{\bar y\nabla\cdot\bm{\beta}}{w}  &= 0,\label{mixed_b}\\
		\innOm{\bar {\bm p}}{\bm r}-\innOm{\bar z}{\nabla\cdot \bm r} &= 0,\label{mixed_c}\\
		\innOm{\nabla\cdot(\bar {\bm p}-\bm{\beta} \bar z)}{w} &= \innOm{\bar y- {y_d}}{w},\label{mixed_d}\\
		\innGa{\gamma \bar u+\bar{\bm p}\cdot\bm n}{v} &=  0,\label{mixed_e}
		\end{align}
		for all $(\bm r,w,v)\in H(\textup{div},\Omega)\times L^2(\Omega)\times L^2(\Gamma)$.
	\end{subequations}
\end{theorem}
\begin{proof}
	The functional $J(u)$ is bounded from below and strictly convex, because  thanks to \Cref{MT28} and \Cref{ML3}, the control-to-state mapping is linear continuous. Using that it is also coercive,  existence of solution follows from the standard argument  of taking a minimizing sequence. Uniqueness of solution follows from the strict convexity.
	
	Since the equation is linear, the functional is differentiable (it is $C^\infty$ indeed) and a standard argument leads to the necessary optimality conditions \eqref{mixed_a}--\eqref{mixed_e}, where the first two equations must be understood in the very weak sense of \Cref{MD26}. Since the problem is strictly convex, these conditions are also sufficient, and therefore the optimality system has a unique solution.
	
	Let us study the regularity of the solution. We already have that $\bar y\in L^2(\Omega)$, so \Cref{ML25} leads in a first step to
	$(\bar z,\bar {\bm p})\in (H^{2}(\Omega)\cap H^1_0(\Omega))\times [H^1(\Omega)]^d$, $\partial_{\bm n}\bar z=\bar {\bm p}\cdot\bm n \in H^{1/2}(\Gamma)$. Noticing that from \eqref{mixed_e} we have that $\bar u=-\gamma^{-1}\bar{\bm p}\cdot\bm n$, it is clear that the optimal control satisfies $\bar u\in H^{1/2}(\Gamma)$. From \Cref{ML3} we deduce $\bar y\in H^{3/2}(\Omega).$
	
	Using the just deduced regularity of $\bar y$ and bootstrapping the argument once, we achieve the desired result.
\end{proof}

\section{HDG Formulation and Implementation}
\label{sec:HDG}

Next, we describe a new HDG method to approximate the solution of the mixed weak form of the optimality system \eqref{mixed_a}--\eqref{mixed_e}.  Throughout this section, we assume $ \Omega $ is a polyhedral domain, not necessarily convex, with $ d \geq 2 $.

Before we introduce the HDG method, we first set some notation.  Let $\mathcal{T}_h$ be a collection of disjoint elements that partition $\Omega$.  We denote by $\partial \mathcal{T}_h$ the set $\{\partial K: K\in \mathcal{T}_h\}$. For an element $K$ of the collection  $\mathcal{T}_h$, let $e = \partial K \cap \Gamma$ denote the boundary face of $ K $ if the $d-1$ Lebesgue measure of $e$ is non-zero. For two elements $K^+$ and $K^-$ of the collection $\mathcal{T}_h$, let $e = \partial K^+ \cap \partial K^-$ denote the interior face between $K^+$ and $K^-$ if the $d-1$ Lebesgue measure of $e$ is non-zero. Let $\varepsilon_h^o$ and $\varepsilon_h^{\partial}$ denote the set of interior and boundary faces, respectively. We denote by $\varepsilon_h$ the union of  $\varepsilon_h^o$ and $\varepsilon_h^{\partial}$. We finally introduce
\begin{align*}
(w,v)_{\mathcal{T}_h} = \sum_{K\in\mathcal{T}_h} (w,v)_K,   \quad\quad\quad\quad\left\langle \zeta,\rho\right\rangle_{\partial\mathcal{T}_h} = \sum_{K\in\mathcal{T}_h} \left\langle \zeta,\rho\right\rangle_{\partial K}.
\end{align*}

Let $\mathcal{P}^k(D)$ denote the set of polynomials of degree at most $k$ on a domain $D$.  We introduce the following discontinuous finite element spaces
\begin{align}
\bm{V}_h  &:= \{\bm{v}\in [L^2(\Omega)]^d: \bm{v}|_{K}\in [\mathcal{P}^k(K)]^d, \forall K\in \mathcal{T}_h\},\\
{W}_h  &:= \{{w}\in L^2(\Omega): {w}|_{K}\in \mathcal{P}^{k+1}(K), \forall K\in \mathcal{T}_h\},\\
{M}_h  &:= \{{\mu}\in L^2(\mathcal{\varepsilon}_h): {\mu}|_{e}\in \mathcal{P}^{k+1}(e), \forall e\in \varepsilon_h\}
\end{align}
for the flux variables, scalar variables, and boundary trace variables, respectively.  Note that the polynomial degree for the scalar and boundary trace variables is one order higher than the polynomial degree for the flux variables.  As far as the authors are aware, this combination of spaces has not been used for an HDG method in the literature.  The boundary trace variables will be used to eliminate the state and flux variables from the coupled global equations, thus substantially reducing the number of degrees of freedom.

Let  $M_h(o)$ and $M_h(\partial)$  denote the subspaces of $M_h$ consisting of $e\in \varepsilon_h^o$ and  $e\in \varepsilon_h^{\partial}$, respectively. Note that $M_h$ consists of functions which are continuous inside the faces (or edges) $e\in \varepsilon_h$ and discontinuous at their borders. In addition, for any functions $w\in W_h$ and $\bm r\in \bm V_h$ we use $\nabla w$ and $ \nabla \cdot \bm r $ to denote the gradient of $ w $ and the divergence of $ \bm r $ taken piecewise on each element $K\in \mathcal T_h$.

\subsection{The HDG Formulation}

To approximate the solution of the mixed weak form \eqref{mixed_a}-\eqref{mixed_e} of the optimality system, the HDG method seeks approximate fluxes ${\bm{q}}_h,{\bm{p}}_h \in \bm{V}_h $, states $ y_h, z_h \in W_h $, interior element boundary traces $ \widehat{y}_h^o,\widehat{z}_h^o \in M_h(o) $, and  boundary control $ u_h \in M_h(\partial)$ satisfying
\begin{subequations}\label{HDG_discrete2}
	\begin{align}
	(\bm q_h,\bm r_1)_{\mathcal T_h}-( y_h,\nabla\cdot\bm r_1)_{\mathcal T_h}+\langle \widehat y_h^o,\bm r_1\cdot\bm n\rangle_{\partial\mathcal T_h\backslash \varepsilon_h^\partial}+\langle  u_h,\bm r_1\cdot\bm n\rangle_{\varepsilon_h^\partial}&=0, \label{HDG_discrete2_a}\\
	-(\bm q_h+\bm \beta y_h,  \nabla w_1)_{\mathcal T_h}-(\nabla\cdot\bm\beta y_h,w_1)_{\mathcal T_h} +\langle\widehat {\bm q}_h\cdot\bm n,w_1\rangle_{\partial\mathcal T_h} \quad \nonumber \\
	+\langle \bm \beta\cdot\bm n  u_h,w_1\rangle_{\varepsilon_h^\partial}   &= ( f, w_1)_{\mathcal T_h},  \label{HDG_discrete2_b}
	\end{align}
	for all $(\bm{r}_1, w_1)\in \bm{V}_h\times W_h$,
	\begin{align}
	(\bm p_h,\bm r_2)_{\mathcal T_h}-(z_h,\nabla\cdot\bm r_2)_{\mathcal T_h}+\langle \widehat z_h^o,\bm r_2\cdot\bm n\rangle_{\partial\mathcal T_h\backslash\varepsilon_h^\partial}&=0,\label{HDG_discrete2_c}\\
	-(\bm p_h-\bm \beta z_h, \nabla w_2)_{\mathcal T_h}+\langle\widehat{\bm p}_h\cdot\bm n,w_2\rangle_{\partial\mathcal T_h} -\langle\bm \beta\cdot\bm n\widehat z_h^o,w_2\rangle_{\partial\mathcal T_h\backslash \varepsilon_h^\partial}\quad \nonumber \\
	 - ( y_h, w_2)_{\mathcal T_h} &=-(y_d, w_2)_{\mathcal T_h}, \label{HDG_discrete2_d}
	\end{align}
	for all $(\bm{r}_2, w_2)\in \bm{V}_h\times W_h$,
	\begin{align}
	\langle\widehat {\bm q}_h\cdot\bm n+\bm \beta\cdot\bm n\widehat y_h^o,\mu_1\rangle_{\partial\mathcal T_h\backslash\varepsilon^{\partial}_h}&=0,\label{HDG_discrete2_e}
	\end{align}
	for all $\mu_1\in M_h(o)$,
	\begin{align}
	\langle\widehat{\bm p}_h\cdot\bm n-\bm \beta\cdot\bm n\widehat z_h^o,\mu_2\rangle_{\partial\mathcal T_h\backslash\varepsilon^{\partial}_h}&=0,\label{HDG_discrete2_f}
	\end{align}
	for all $\mu_2\in M_h(o)$, and the optimality condition
	\begin{align}
	\langle u_h, \mu_3 \rangle_{{\varepsilon_h^{\partial}}}+\langle \gamma^{-1}\widehat{\bm{p}}_h\cdot \bm{n}, \mu_3\rangle_{{{\varepsilon_h^{\partial}}}} &=0, \label{HDG_discrete2_g}
	\end{align}
	for all $\mu_3\in M_h(\partial)$.
	
	The numerical traces on $\partial\mathcal{T}_h$ are defined as
	\begin{align}
	\widehat{\bm{q}}_h\cdot \bm n &=\bm q_h\cdot\bm n+h^{-1} (y_h-\widehat y_h^o) + \tau_1(y_h-\widehat y_h^o)   \qquad ~~~ \mbox{on} \; \partial \mathcal{T}_h\backslash\varepsilon_h^\partial, \label{HDG_discrete2_h}\\
	\widehat{\bm{q}}_h\cdot \bm n &=\bm q_h\cdot\bm n+h^{-1} (y_h-u_h) +  \tau_1(y_h-u_h)  \quad~~~~~~ \mbox{on}\;  \varepsilon_h^\partial, \label{HDG_discrete2_i}\\
	\widehat{\bm{p}}_h\cdot \bm n &=\bm p_h\cdot\bm n+h^{-1}(z_h-\widehat z_h^o) + \tau_2(z_h-\widehat y_h^o) \qquad~~~ \mbox{on} \; \partial \mathcal{T}_h\backslash\varepsilon_h^\partial,\label{HDG_discrete2_j}\\
	\widehat{\bm{p}}_h\cdot \bm n &=\bm p_h\cdot\bm n+h^{-1} z_h+\tau_2 z_h\qquad\qquad\qquad\qquad~\mbox{on}\;  \varepsilon_h^\partial,\label{HDG_discrete2_k}
	\end{align}
\end{subequations}
where $\tau_1$ and $\tau_2$ are stabilization functions defined on $\partial \mathcal T_h$.   This completes the formulation of the HDG method.

To guarantee the stability for existing HDG methods, the stabilization functions $\tau_1$ and $\tau_2$ for the (uncoupled) convection diffusion equation and the dual problem are chosen to satisfy
$$
\tau_1  \geq  \frac{1}{2} \bm \beta \cdot \bm n,  \quad  \tau_2  \geq  -\frac{1}{2} \bm \beta \cdot \bm n
$$
on $ \partial \mathcal{T}_h $; see, e.g., \cite{ChenCockburn12,ChenCockburn14,MR3440284,FuQiuZhang15}.  However, in our convergence analysis in \Cref{sec:analysis} for the fully coupled optimality system we require the stabilization functions to be chosen very specifically.  These requirements on the stabilization functions arise naturally in our analysis.

\subsection{Implementation}
\label{sec:implementation}
For the HDG implementation, we proceed similarly to our earlier work \cite{HuShenSinglerZhangZheng_HDG_Dirichlet_control1}.  A fundamental aspect of the HDG method is the local solver, which reduces the number of globally coupled unknowns.  The standard approach is to implement the local solver element-by-element independently and then assemble the global system.  As in \cite{HuShenSinglerZhangZheng_HDG_Dirichlet_control1}, here we first assemble a large global system and then reduce the size of the system using simple block-diagonal matrix operations.  This process is equivalent to the standard approach.

Substitute \eqref{HDG_discrete2_h}-\eqref{HDG_discrete2_k} into \eqref{HDG_discrete2_a}-\eqref{HDG_discrete2_g} and perform some simple manipulations to obtain
\[ ({\bm{q}}_h,{\bm{p}}_h, y_h,z_h,{\widehat{y}}_h^o,{\widehat{z}}_h^o,u_h)\in \bm{V}_h\times\bm{V}_h\times W_h \times W_h\times M_h(o)\times M_h(o)\times M_h(\partial) \]
is the solution of the following weak formulation:
\begin{subequations}\label{imple}
	\begin{align}
	(\bm{q}_h, \bm{r_1})_{{\mathcal{T}_h}}- (y_h, \nabla\cdot \bm{r_1})_{{\mathcal{T}_h}}+\langle \widehat{y}_h^o, \bm{r_1}\cdot \bm{n} \rangle_{\partial{{\mathcal{T}_h}}\backslash \varepsilon_h^{\partial}} + \langle u_h, \bm{r_1}\cdot \bm{n} \rangle_{\varepsilon_h^{\partial}} &= 0 , \label{imple_a}\\
	(\bm{p}_h, \bm{r_2})_{{\mathcal{T}_h}}- (z_h, \nabla\cdot \bm{r_2})_{{\mathcal{T}_h}}+\langle \widehat{z}_h^o, \bm{r_2}\cdot \bm{n} \rangle_{\partial{{\mathcal{T}_h}}\backslash \varepsilon_h^{\partial}} &=0, \label{imple_b}\\
	(\nabla\cdot\bm{q}_h,  w_1)_{{\mathcal{T}_h}} - (\bm{\beta} y_h, \nabla w_1)_{\mathcal T_h} - (\nabla\cdot\bm{\beta} y_h, w_1)_{\mathcal T_h} \quad &\nonumber\\
	\quad + \langle (h^{-1}+\tau_1)  y_h, w_1\rangle_{\partial\mathcal T_h}  +\langle  (\bm{\beta}\cdot\bm n - \tau_1 - h^{-1})\widehat y_h^o, w_1 \rangle_{\partial{{\mathcal{T}_h}}\backslash\varepsilon_h^\partial}  \quad &\nonumber\\ 
	+ \langle  (\bm{\beta}\cdot\bm n - \tau_1-h^{-1}) u_h, w_1 \rangle_{\varepsilon_h^\partial}&=(f, w_1)_{{\mathcal{T}_h}}, \label{imple_c}\\
	(\nabla\cdot\bm{p}_h,  w_2)_{{\mathcal{T}_h}} - (y_h, w_2)_{\mathcal T_h} + (\bm{\beta} z_h, \nabla w_2)_{\mathcal T_h} + \langle (h^{-1}+\tau_2) z_h, w_2\rangle_{\partial \mathcal T_h} \quad &\nonumber\\
 -\langle (h^{-1}+\tau_2 + \bm{\beta}\cdot \bm n) \widehat{z}_h^o, w_2 \rangle_{\partial{{\mathcal{T}_h}}\backslash \varepsilon_h^{\partial}} &=- (y_d, w_2)_{{\mathcal{T}_h}} , \label{imple_d}\\
	\langle{\bm{q}_h}\cdot \bm{n}, \mu_1 \rangle_{\partial\mathcal{T}_{h}\backslash {\varepsilon_h^{\partial}}}+\langle (h^{-1}+\tau_1) y_h,\mu_1 \rangle_{\partial\mathcal{T}_{h}\backslash {\varepsilon_h^{\partial}}} \quad &\nonumber\\
	+\langle (\bm{\beta}\cdot\bm n - \tau_1-h^{-1})\widehat{y}_h^{o} ,\mu_1 \rangle_{\partial\mathcal{T}_{h}\backslash {\varepsilon_h^{\partial}}}&=0, \label{imple_e}\\
	\langle {\bm{p}_h}\cdot \bm{n}, \mu_2\rangle_{ \partial\mathcal{T}_{h}\backslash {\varepsilon_h^{\partial}}}+\langle (h^{-1}+\tau_2) z_h, \mu_2\rangle_{ \partial\mathcal{T}_{h}\backslash {\varepsilon_h^{\partial}}} \quad &\nonumber\\
	-\langle (\bm{\beta}\cdot\bm n+\tau_2+h^{-1}) \widehat{z}_h^o, \mu_2\rangle_{ \partial\mathcal{T}_{h}\backslash {\varepsilon_h^{\partial}}} &=0,
	\label{imple_f}\\
	\langle {\bm{p}_h}\cdot \bm{n}, \mu_3\rangle_{{\varepsilon_h^{\partial}}} + \gamma\langle u_h, \mu_3 \rangle_{\varepsilon_h^{\partial}}+\langle (h^{-1}+\tau_2) z_h, \mu_3\rangle_{{\varepsilon_h^{\partial}}}&=0, \label{imple_g}
	\end{align}
\end{subequations}
for all  $({\bm{r}_1},{\bm{r}_2},w_1,w_2,\mu_1,\mu_2,\mu_3)\in \bm{V}_h\times\bm{V}_h\times W_h \times W_h\times M_h(o)\times M_h(o)\times M_h(\partial)$.

For $\bm{V}_h = \mbox{span}\{\bm\varphi_i\}_{i=1}^{N_1}$, $W_h=\mbox{span}\{\phi_i\}_{i=1}^{N_2}$, $M_h^{o}=\mbox{span}\{\psi_i\}_{i=1}^{N_3} $, and $M_h^{\partial}=\mbox{span}\{\psi_i\}_{i=1+N_3}^{N_4}$, assume
\begin{equation}\label{expre}
\begin{split}
&\bm q_{h}= \sum_{j=1}^{N_1}q_{j}\bm\varphi_j,  \quad \bm p_{h} =  \sum_{j=1}^{N_1}p_{j}\bm\varphi_j \quad  y_h = \sum_{j=1}^{N_2}y_{j}\phi_j \quad z_h =  \sum_{j=1}^{N_2} z_{j}\phi_j,\\
&  \widehat{y}_h^o = \sum_{j=1}^{N_3}\alpha_{j}\psi_{j},\quad  \widehat{z}_h^o = \sum_{j=1}^{N_3}\gamma_{j}\psi_{j}, \quad u_h = \sum_{j=1+N_3}^{N_4}\beta_{j}\psi_{j}.
\end{split}
\end{equation}
Substitute \eqref{expre} into \eqref{imple_a}-\eqref{imple_f} and use the corresponding  test functions to test \eqref{imple_a}-\eqref{imple_f}, respectively, to obtain the matrix equation
\begin{align}\label{system_equation}
\begin{bmatrix}
A_1  &0 &-A_2&0  & A_{20}&0&A_{21} \\
0 & A_1 &0 &-A_2& 0  &A_{20}& 0 \\
A_2^T &0&A_{18}&0&A_{22}&0& A_{23}\\
0 &A_2^T & -A_{12} &A_{19} &0&A_{24} &0\\
A_{20}^T & 0 &A_{25}&0 &A_{26}&0&0  \\
0& A_{20}^T &0&A_{24} &0&A_{25} &0 \\
0 & A_{26}&0 &A_{27}&0&0 & \gamma A_{28}
\end{bmatrix}
\left[ {\begin{array}{*{20}{c}}
	\mathfrak{q}\\
	\mathfrak{p}\\
	\mathfrak{y}\\
	\mathfrak{z}\\
	\mathfrak{\widehat y}\\
	\mathfrak{\widehat z}\\
	\mathfrak{u}
	\end{array}} \right]
=\left[ {\begin{array}{*{20}{c}}
	0\\
	0\\
	b_1\\
	-b_2\\
	0\\
	0\\
	0\\
	\end{array}} \right],
\end{align}
where $\mathfrak{q},\mathfrak{p},\mathfrak{y},\mathfrak{z},\mathfrak{\widehat y},\mathfrak{\widehat z},\mathfrak{u}$ are the coefficient vectors for $\bm q_h,\bm p_h,y_h,z_h,\widehat y_h^o, \widehat z_h^o, u_h$, respectively, and
\begin{gather*}
A_1= [(\bm\varphi_j,\bm\varphi_i )_{\mathcal{T}_h}],  \qquad \qquad\quad A_2 = [(\phi_j,\nabla\cdot\bm{\varphi_i})_{\mathcal{T}_h}],  \qquad\qquad \quad A_3 = [(\psi_j,\bm{\varphi}_i\cdot \bm n)_{\mathcal{T}_h}],\\ 
  A_4 = [(\bm{\beta}\phi_j,\nabla\phi_i)_{\mathcal{T}_h}],\quad\qquad
A_5 = [(\nabla\cdot\bm{\beta}\phi_j, \phi_i)_{\mathcal{T}_h}],\quad\qquad A_6 = [(\bm\beta\phi_j,\nabla\phi_i)_{\mathcal{T}_h}],\\
  A_7= [\langle  h^{-1}\phi_j, \phi_i \rangle_{\partial{{\mathcal{T}_h}}}],\quad  \quad A_8 = [\langle  \tau_1\phi_j, \phi_i \rangle_{\partial{{\mathcal{T}_h}}}], \quad\quad
A_9= [\langle  \tau_2 \phi_j, \phi_i \rangle_{\partial{{\mathcal{T}_h}}}],\\
  A_{10}= [\langle  \bm{\beta}\cdot\bm n\psi_j, \phi_i \rangle_{\partial{{\mathcal{T}_h}}}],
~  A_{11}=  [\left\langle \tau_1\psi_j,{\phi_i}\right\rangle_{\partial\mathcal{T}_h}],
~ A_{12}=  [\left\langle h^{-1}\psi_j,{\phi_i}\right\rangle_{\partial\mathcal{T}_h}], \\
A_{13} =  [(\phi_j,\phi_i)_{\mathcal{T}_h}],~~ A_{14}=
[\left\langle \tau_1\psi_j,\psi_i\right\rangle_{\partial\mathcal{T}_h}], ~~ A_{15}=
[\left\langle h^{-1}\psi_j,\psi_i\right\rangle_{\partial\mathcal{T}_h}],\\
 A_{16}= [\left\langle \bm{\beta}\cdot\bm n \psi_j,\psi_i\right\rangle_{\partial\mathcal{T}_h}], \qquad  \qquad  A_{17}= [\left\langle \tau_2 \psi_j,\psi_i\right\rangle_{\partial\mathcal{T}_h}],\\
  A_{18}= A_7+A_8-A_4-A_5,~ A_{19} = A_4+ A_7+A_9,\\
b_1 = [(f,\phi_i )_{\mathcal{T}_h}], \quad  b_2 = [(y_d,\phi_i )_{\mathcal{T}_h}].
\end{gather*}
The remaining matrices are constructed by  extracting the corresponding rows and columns from linear combinations of $A_3 $ to  $A_{17}$.

Equation \eqref{system_equation} can be rewritten as
\begin{align}\label{system_equation2}
\begin{bmatrix}
B_1 & B_2&B_3\\
-B_2^T & B_4&B_5\\
B_6&B_7&B_8\\
\end{bmatrix}
\left[ {\begin{array}{*{20}{c}}
	\bm{\alpha}\\
	\bm{\beta}\\
	\bm{\gamma}
	\end{array}} \right]
=\left[ {\begin{array}{*{20}{c}}
	0\\
	b\\
	0
	\end{array}} \right],
\end{align}
where $\bm{\alpha}=[\mathfrak{q};\mathfrak{p}]$, $\bm{\beta}=[\mathfrak{y};\mathfrak{z}]$, $\bm{\gamma}=[\mathfrak{\widehat y};\mathfrak{\widehat z};\mathfrak{u}]$, $ b = [ b_1; -b_2 ] $, and $\{B_i\}_{i=1}^8$ are the corresponding blocks of the coefficient matrix in \eqref{system_equation}.

As in \cite{HuShenSinglerZhangZheng_HDG_Dirichlet_control1}, we use the first two equations of  \eqref{system_equation2} to solve for $\bm{\alpha}$ and $\bm{\beta}$ using simple and efficient block-diagonal matrix computations.  Eliminating $\bm{\alpha}$ and $\bm{\beta}$ gives a reduced globally coupled equation for $\bm{\gamma}$ only:
\begin{align}\label{global_eq}
\mathbb{K}\bm{\gamma} = \mathbb{F}.
\end{align}
Note that the globally coupled system only involves the vector $ \bm \gamma $, which contains the coefficients of the approximate boundary traces.  Therefore, the number of globally coupled degrees of freedom is much smaller than the total number of degrees of freedom for all variables.

The details of the above procedure are similar to our previous work \cite{HuShenSinglerZhangZheng_HDG_Dirichlet_control1}.  We only need to show
the following result.
\begin{prop}
	If $ \min{(\tau_1-\frac 1 2 \bm \beta \cdot \bm n)}|_{\partial K} >0 $ and $\min{(\tau_2+\frac 1 2 \bm \beta \cdot \bm n)}|_{\partial K} >0$ for any $K\in\mathcal T_h$, then the matrices $A_{18}$ and $A_{19}$ in \eqref{system_equation} are positive definite.
\end{prop}
\begin{proof}
	We only prove $A_{18}$ is positive definite; a similar argument applies to $A_{19}$.  The matrix $A_{18}$ is positive definite if and only if $\bm x^TA_{18}\bm x>0$ for any $\bm x=[x_1,x_2,\cdots,x_{N_2}]\in\mathbb R^{N_2} $. For $ x = \sum_{j=1}^{N_2} x_j \phi_j$, we have
	\begin{align*}
	\bm x^T A_{18}\bm x = \langle h^{-1}x, x\rangle_{\partial \mathcal T_h}  + \langle \tau_1x, x\rangle_{\partial \mathcal T_h} - (\bm{\beta}x, \nabla x)_{\mathcal T_h} - (\nabla\cdot\bm{\beta} x, x)_{\mathcal T_h}.
	\end{align*}
	Moreover,
	\begin{align*}
	(\bm \beta x,\nabla x)_{\mathcal T_h}&=(\bm \beta\cdot\nabla x,x)_{\mathcal T_h}=(\nabla\cdot(\bm \beta x),x)_{\mathcal T_h}-(\nabla\cdot\bm \beta x,x)_{\mathcal T_h}\\
	&=\langle\bm \beta\cdot\bm n x,x\rangle_{\partial\mathcal T_h}-(\bm \beta x,\nabla x)_{\mathcal T_h}-(\nabla\cdot\bm \beta x,x)_{\mathcal T_h},
	\end{align*}
	which implies
	\begin{align*}
	(\bm \beta x,\nabla x)_{\mathcal T_h}&=\frac12\langle\bm \beta\cdot\bm n x,x\rangle_{\partial\mathcal T_h}-\frac12(\nabla\cdot\bm \beta x,x)_{\mathcal T_h}.
	\end{align*}
	Therefore,
	\begin{align*}
	\bm x^T A_{18}\bm x  = \langle (h^{-1} + \tau_1-\frac 1 2 \bm{\beta}\cdot\bm n) x, x\rangle_{\partial\mathcal T_h} - \frac 1 2(\nabla\cdot \bm{\beta}x, x)_{\mathcal T_h}>0,
	\end{align*}
	by the assumption concerning $ \tau_1 $ and the condition $\nabla\cdot\bm \beta\le 0$.
\end{proof}

\section{Error Analysis}
\label{sec:analysis}

Next, we provide a convergence analysis of the above HDG method for the Dirichlet boundary control problem.  We assume the solution of the optimality system has certain regularity properties.  In 2D, due to the theoretical results in \Cref{sec:Analysis_of_the_Dirichlet_Control_Problem}, we can give simple conditions that guarantee the unique solution has the necessary regularity.  In 3D, we lack the necessary regularity theory; however, our convergence results still apply if there exists a unique solution of the optimality system with the required regularity.

We begin with a precise statement of our assumptions and the main convergence result.

\subsection{Assumptions and Main Result}

Throughout this section, we assume $\Omega$ is a bounded convex polyhedral domain.  We assume throughout that $ \bm \beta $ satisfies
\begin{equation}\label{eqn:beta_assumptions2}
\bm \beta \in [C(\overline{\Omega})]^d,  \quad  \nabla\cdot\bm{\beta}\in L^\infty(\Omega),  \quad  \nabla \cdot \bm \beta \leq 0,  \quad  \nabla \nabla \cdot \bm{\beta}\in[L^2(\Omega)]^d.
\end{equation}
Note that this condition is slightly stronger than the condition \eqref{eqn:beta_assumptions1} made for the analysis in \Cref{sec:Analysis_of_the_Dirichlet_Control_Problem}.  Here, we assume $ \bm \beta $ is continuous on $ \overline{\Omega} $, while before we assumed $ \bm \beta \in [L^\infty({\Omega})]^d $.

For our theoretical results, we choose the stabilization functions $\tau_1$ and $\tau_2$ to satisfy
\begin{description}
	
	\item[\textbf{(A1)}] $\tau_2$ is piecewise constant on $\partial \mathcal T_h$.
	
	\item[\textbf{(A2)}] $\tau_1 = \tau_2 + \bm{\beta}\cdot \bm n$.
	
	\item[\textbf{(A3)}] For any  $K\in\mathcal T_h$, $\min{(\tau_2+\frac 1 2 \bm \beta \cdot \bm n)}|_{\partial K} >0$.
	
\end{description}
We note that \textbf{(A2)} and \textbf{(A3)} imply
\begin{equation}\label{eqn:tau1_condition}
\min{(\tau_1-\frac 1 2 \bm \beta \cdot \bm n)}|_{\partial K} >0  \quad  \mbox{for any $K\in\mathcal T_h$.}
\end{equation}
In our analysis, we use the conditions \textbf{(A3)} and \eqref{eqn:tau1_condition} frequently and therefore we rarely mention them explicitly.  However, we use \textbf{(A1)} and \textbf{(A2)}  less frequently, and therefore we typically mention these conditions when we use them.

We also assume throughout that there exists a unique solution of the optimality system \eqref{mixed_a}--\eqref{mixed_e} that satisfies
\begin{equation}\label{eqn:reg_assumption1}
y \in H^{r_y}(\Omega),  \quad  z \in H^{r_z}(\Omega) \cap H^1_0(\Omega),  \quad  \bm q \in [ H^{r_{\bm q}}(\Omega) ]^d,  \quad   \bm p \in [ H^{r_{\bm p}}(\Omega) ]^d,
\end{equation}
where
\begin{equation}\label{eqn:reg_assumption2}
r_y > 1,  \quad  r_z > 2,  \quad  r_{\bm q} > 1/2,  \quad  r_{\bm p} > 1.
\end{equation}
This regularity condition ensures that the convergence rates in \Cref{main_res} below are positive for all variables.  

We note that we require $ r_{\bm q} > 1/2 $ (instead of $ r_{\bm q} > 0 $) here in order to guarantee $ \bm q $ has a well-defined boundary trace in $ L^2(\Gamma) $.  We use this property in our analysis.  As mentioned in the introduction, we relax this assumption in the second part of this work and only require $ r_{\bm q} > 0 $.  Dealing with the very low regularity of $ \bm q $ requires entirely different HDG analysis techniques than we use here.

In the 2D case, simple conditions on the desired state $ y_d $ and the domain $ \Omega $ guarantee that the solution has the above regularity; see \Cref{cor:main_result} below.  In the 3D case, we do not have theory that gives simple conditions guaranteeing such solutions exist.

We now state our main convergence result.
\begin{theorem}\label{main_res}
	Let
	\begin{equation}\label{eqn:s_rates}
	\begin{split}
	s_{\bm q} &= \min\{r_{\bm q}, k+1\}, \qquad  s_{y} = \min\{r_{y}, k+2\}, \\
	s_{\bm p} &= \min\{r_{\bm p}, k+1\},  \qquad s_{z} = \min\{r_{z}, k+2\}.
	\end{split}
	\end{equation}
	If the above assumptions hold, then
	\begin{align*}
	\norm{u-u_h}_{\varepsilon_h^\partial}&\lesssim h^{s_{\bm p}-\frac 1 2}\norm{\bm p}_{s_{\bm p},\Omega} +  h^{s_{z}-\frac 3 2}\norm{z}_{s_{z},\Omega} + h^{s_{\bm q}+\frac 1 2}\norm{\bm q}_{s_{\bm q},\Omega} + h^{s_{y}-\frac 12 }\norm{y}_{s_{y},\Omega},\\
	\norm {y-y_h}_{\mathcal T_h} &\lesssim h^{s_{\bm p}-\frac 1 2}\norm{\bm p}_{s_{\bm p},\Omega} +  h^{s_{z}-\frac 3 2}\norm{z}_{s_{z},\Omega} + h^{s_{\bm q}+\frac 1 2}\norm{\bm q}_{s_{\bm q},\Omega} + h^{s_{y}-\frac 12 }\norm{y}_{s_{y},\Omega},\\
	\norm {\bm p - \bm p_h}_{\mathcal T_h}  &\lesssim h^{s_{\bm p}-\frac 1 2}\norm{\bm p}_{s_{\bm p},\Omega} +  h^{s_{z}-\frac 3 2}\norm{z}_{s_{z},\Omega} + h^{s_{\bm q}+\frac 1 2}\norm{\bm q}_{s_{\bm q},\Omega} + h^{s_{y}-\frac 12 }\norm{y}_{s_{y},\Omega},\\
	\norm {z - z_h}_{\mathcal T_h} & \lesssim  h^{s_{\bm p}-\frac 1 2}\norm{\bm p}_{s_{\bm p},\Omega} +  h^{s_{z}-\frac 3 2}\norm{z}_{s_{z},\Omega} + h^{s_{\bm q}+\frac 1 2}\norm{\bm q}_{s_{\bm q},\Omega} + h^{s_{y}-\frac 12 }\norm{y}_{s_{y},\Omega}.
	\end{align*}
	If in addition $ k \geq 1 $, then
	\begin{align*}
	\norm {\bm q - \bm q_h}_{\mathcal T_h} &\lesssim h^{s_{\bm p}-1}\norm{\bm p}_{s_{\bm p},\Omega} +  h^{s_{z}-2}\norm{z}_{s_{z},\Omega} + h^{s_{\bm q}}\norm{\bm q}_{s_{\bm q},\Omega} + h^{s_{y}-1 }\norm{y}_{s_{y},\Omega}.\\
	\end{align*}
\end{theorem}

Now we specialize to the 2D case.  For a convex polygonal domain $ \Omega $, let $ \omega $ denote its largest interior angle.  As mentioned before, $ \omega $ must satisfy $ 1 < \pi/\omega \leq 3 $, i.e., $ \omega \in [\pi/3,\pi) $.  The limiting regularity condition is $ r_{\bm q} > 1/2 $.  Therefore, to guarantee the regularity condition \eqref{eqn:reg_assumption1}-\eqref{eqn:reg_assumption2}, by \Cref{MT210} we need two conditions:
\begin{enumerate}
	\item  $ \pi/\omega - 1/2 > 1 $, i.e., $ \omega < 2\pi/3 $, and
	\item  $ y_d \in H^{t^*}(\Omega) $ for some $ t^* \in (1/2,1) $.
\end{enumerate}
As mentioned earlier, we remove these restrictions in the second part of this work.

Applying \Cref{MT210} and the main theorem above gives the following result.
\begin{corollary}\label{cor:main_result}
	Suppose $ d = 2 $, $ f = 0 $, and $ y_d \in H^{t^*}(\Omega) $ for some $ t^* \in (1/2,1) $.  Let $ \omega \in ( \pi/3, 2\pi/3 ) $ be the largest interior angle of $\Gamma$, and define $ r_{\Omega}$ by
	\[
	r_{\Omega} = \min\left\{ \frac{3}{2}, \frac{\pi}{\omega} - \frac{1}{2}, t^* + \frac{1}{2} \right\} \in (1, 3/2).
	\]
	Then the regularity condition \eqref{eqn:reg_assumption1}-\eqref{eqn:reg_assumption2} is satisfied.  Also, if $ k = 1 $, then for any $ r < r_\Omega $ we have
	\begin{align*}
	\norm{u-u_h}_{\varepsilon_h^\partial}&\lesssim h^{r} (\norm{\bm p}_{H^{r+1/2}(\Omega)} +  \norm{z}_{H^{r+3/2}(\Omega)} + \norm{\bm q}_{H^{r-1/2}(\Omega)} + \norm{y}_{H^{r+1/2}(\Omega)}),\\
	\norm{y-y_h}_{\mathcal T_h}&\lesssim h^{r} (\norm{\bm p}_{H^{r+1/2}(\Omega)} +  \norm{z}_{H^{r+3/2}(\Omega)} + \norm{\bm q}_{H^{r-1/2}(\Omega)} + \norm{y}_{H^{r+1/2}(\Omega)}),\\
	\norm {\bm q - \bm q_h}_{\mathcal T_h}  &\lesssim h^{r-1/2} (\norm{\bm p}_{H^{r+1/2}(\Omega)} +  \norm{z}_{H^{r+3/2}(\Omega)} + \norm{\bm q}_{H^{r-1/2}(\Omega)} + \norm{y}_{H^{r+1/2}(\Omega)}),\\
	\norm {\bm p - \bm p_h}_{\mathcal T_h}   &\lesssim h^{r} (\norm{\bm p}_{H^{r+1/2}(\Omega)} +  \norm{z}_{H^{r+3/2}(\Omega)} + \norm{\bm q}_{H^{r-1/2}(\Omega)} + \norm{y}_{H^{r}+1/2(\Omega)}),\\
	\norm {z - z_h}_{\mathcal T_h}   & \lesssim  h^{r} (\norm{\bm p}_{H^{r+1/2}(\Omega)} +  \norm{z}_{H^{r+3/2}(\Omega)} + \norm{\bm q}_{H^{r-1/2}(\Omega)} + \norm{y}_{H^{r+1/2}(\Omega)}).
	\end{align*}
	Furthermore, if $ k = 0 $, then for any $ r \in (1,r_\Omega) $ we have
	\begin{align*}
	\norm{u-u_h}_{\varepsilon_h^\partial}&\lesssim h^{1/2} (\norm{\bm p}_{H^{1}(\Omega)} +  \norm{z}_{H^{2}(\Omega)} + \norm{\bm q}_{H^{r-1/2}(\Omega)} + \norm{y}_{H^{r+1/2}(\Omega)}),\\
	\norm{y-y_h}_{\mathcal T_h}&\lesssim h^{1/2} (\norm{\bm p}_{H^{1}(\Omega)} +  \norm{z}_{H^{2}(\Omega)} + \norm{\bm q}_{H^{r-1/2}(\Omega)} + \norm{y}_{H^{r+1/2}(\Omega)}),\\
	\norm {\bm p - \bm p_h}_{\mathcal T_h}   &\lesssim h^{1/2} (\norm{\bm p}_{H^{1}(\Omega)} +  \norm{z}_{H^{2}(\Omega)} + \norm{\bm q}_{H^{r-1/2}(\Omega)} + \norm{y}_{H^{r+1/2}(\Omega)}),\\
	\norm {z - z_h}_{\mathcal T_h}   & \lesssim  h^{1/2} (\norm{\bm p}_{H^{1}(\Omega)} +  \norm{z}_{H^{2}(\Omega)} + \norm{\bm q}_{H^{r-1/2}(\Omega)} + \norm{y}_{H^{r+1/2}(\Omega)}).
	\end{align*}
\end{corollary}
\Cref{MT210} gives $ u \in H^r(\Gamma) $, and so the convergence rate for the control is optimal for $ k = 1 $.  Similarly, the convergence rate for the flux $ \bm q $ is optimal for $ k = 1 $.  The convergence rates are suboptimal for the other variables when $ k = 1 $ and for all variables when $ k = 0 $.

Since $ r_\Omega \in (1,3/2) $, when $ k = 1 $ this result guarantees a superlinear convergence rate for all variables except $ \bm{q} $.  Also, if $ \Omega $ is a rectangle (i.e., $ \omega = \pi/2 $), $y_d\in H^{1- \varepsilon}(\Omega) $, and $ k = 1 $, then $ r_\Omega = 3/2- \varepsilon $ and therefore for any $ \varepsilon > 0 $ all variables except $ \bm q $ converge at the rate $ O(h^{3/2 - \varepsilon}) $, and $ \bm q $ converges at the rate $ O(h^{1 - \varepsilon}) $.

\subsection{Preliminary material}
\label{sec:Projectionoperator}

Next, we discuss $ L^2 $ projections, HDG operators $ \mathscr B_1 $ and $ \mathscr B_2 $, and the well-posedness of the HDG equations.

We first define the standard $L^2$ projections $\bm\Pi :  [L^2(\Omega)]^d \to \bm V_h$, $\Pi :  L^2(\Omega) \to W_h$, and $P_M:  L^2(\varepsilon_h) \to M_h$, which satisfy
\begin{equation}\label{L2_projection}
\begin{split}
(\bm\Pi \bm q,\bm r)_{K} &= (\bm q,\bm r)_{K} ,\qquad \forall \bm r\in [{\mathcal P}_{k}(K)]^d,\\
(\Pi y,w)_{K}  &= (y,w)_{K} ,\qquad \forall w\in \mathcal P_{k+1}(K),\\
\left\langle P_M m, \mu\right\rangle_{ e} &= \left\langle  m, \mu\right\rangle_{e }, \quad\;\;\; \forall \mu\in \mathcal P_{k+1}(e).
\end{split}
\end{equation}
In the analysis, we use the following classical results:
\begin{subequations}\label{classical_ine}
	\begin{align}
	\norm {\bm q -\bm{\Pi q}}_{\mathcal T_h} &\lesssim h^{s_{\bm q}} \norm{\bm q}_{s_{\bm q},\Omega},\quad\norm {y -{\Pi y}}_{\mathcal T_h} \lesssim h^{s_{y}} \norm{y}_{s_{y},\Omega},\\
	\norm {y -{\Pi y}}_{\partial\mathcal T_h} &\lesssim h^{s_{y}-\frac 1 2} \norm{y}_{s_{y},\Omega},
	\quad\norm {\bm q\cdot \bm n -\bm{\Pi q}\cdot \bm n}_{\partial \mathcal T_h} \lesssim h^{s_{\bm q}-\frac 12} \norm{\bm q}_{s_{\bm q},\Omega},\\
	\norm {w}_{\partial \mathcal T_h} &\lesssim h^{-\frac 12} \norm {w}_{ \mathcal T_h}, \:\quad \forall w\in W_h,
	\end{align}
\end{subequations}
We have the same projection error bounds for $\bm p$ and $z$.

We define the following HDG operators $ \mathscr B_1$ and $ \mathscr B_2 $.
\begin{align}
\hspace{3em}&\hspace{-3em} \mathscr  B_1( \bm q_h,y_h,\widehat y_h^o;\bm r_1,w_1,\mu_1) \nonumber\\
&=(\bm q_h,\bm r_1)_{\mathcal T_h}-( y_h,\nabla\cdot\bm r_1)_{\mathcal T_h}+\langle \widehat y_h^o,\bm r_1\cdot\bm n\rangle_{\partial\mathcal T_h\backslash \varepsilon_h^\partial}\nonumber\\
&\quad-(\bm q_h+\bm \beta y_h,  \nabla w_1)_{\mathcal T_h}-(\nabla\cdot\bm\beta y_h,w_1)_{\mathcal T_h}\nonumber\\
&\quad+\langle {\bm q}_h\cdot\bm n +(h^{-1}+\tau_1)y_h, w_1\rangle_{\partial\mathcal T_h}+\langle (\bm\beta\cdot\bm n -h^{-1}-\tau_1) \widehat y_h^o,w_1\rangle_{\partial\mathcal T_h\backslash \varepsilon_h^\partial}\nonumber\\
&\quad -\langle  {\bm q}_h\cdot\bm n+\bm \beta\cdot\bm n\widehat y_h^o +(h^{-1}+\tau_1)(y_h-\widehat y_h^o), \mu_1\rangle_{\partial\mathcal T_h\backslash\varepsilon^{\partial}_h},\label{def_B1}\\
\hspace{3em}&\hspace{-3em} \mathscr B_2 (\bm p_h,z_h,\widehat z_h^o;\bm r_2, w_2,\mu_2)\nonumber\\
& =(\bm p_h,\bm r_2)_{\mathcal T_h}-( z_h,\nabla\cdot\bm r_2)_{\mathcal T_h}+\langle \widehat z_h^o,\bm r_2\cdot\bm n\rangle_{\partial\mathcal T_h\backslash\varepsilon_h^\partial}-(\bm p_h-\bm \beta z_h,  \nabla w_2)_{\mathcal T_h}\nonumber\\
&\quad+\langle {\bm p}_h\cdot\bm n +(h^{-1}+\tau_2) z_h, w_2\rangle_{\partial\mathcal T_h} -\langle (\bm \beta\cdot\bm n + h^{-1}+\tau_2)\widehat z_h^o ,w_2\rangle_{\partial\mathcal T_h\backslash\varepsilon_h^\partial}\nonumber\\
&\quad-\langle  {\bm p}_h\cdot\bm n-\bm \beta\cdot\bm n\widehat z_h^o +(h^{-1}+\tau_2)(z_h-\widehat z_h^o), \mu_2\rangle_{\partial\mathcal T_h\backslash\varepsilon^{\partial}_h}\label{def_B2}.
\end{align}
By the definition of $\mathscr B_1$ and $\mathscr B_2$,  we can rewrite the HDG formulation of the optimality system \eqref{HDG_discrete2}, as follows: find $({\bm{q}}_h,{\bm{p}}_h,y_h,z_h,\widehat y_h^o,\widehat z_h^o,u_h)\in \bm{V}_h\times\bm{V}_h\times W_h \times W_h\times M_h(o)\times M_h(o)\times M_h(\partial)$  such that
\begin{subequations}\label{HDG_full_discrete}
	\begin{align}
	\mathscr B_1(\bm q_h,y_h,\widehat y_h^o;\bm r_1,w_1,\mu_1)&=( f, w_1)_{\mathcal T_h} - \langle u_h,\bm r_1\cdot\bm n\rangle_{\varepsilon_h^\partial} \nonumber \\
	&\quad-\langle (\bm\beta\cdot\bm n-h^{-1}-\tau_1) u_h, w_1\rangle_{\varepsilon_h^\partial},\label{HDG_full_discrete_a}\\
	\mathscr B_2(\bm p_h,z_h,\widehat z_h^o;\bm r_2,w_2,\mu_2)&=(y_h-y_d,w_2)_{\mathcal T_h},\label{HDG_full_discrete_b}\\
	\gamma^{-1}\langle {\bm{p}}_h\cdot \bm{n} + h^{-1} z_h + \tau_2 z_h, \mu_3\rangle_{{{\varepsilon_h^{\partial}}}} &= -\langle u_h, \mu_3 \rangle_{{\varepsilon_h^{\partial}}}, \label{HDG_full_discrete_e}
	\end{align}
\end{subequations}
for all $\left(\bm{r}_1, \bm{r}_2, w_1, w_2, \mu_1, \mu_2, {\mu}_3\right)\in \bm V_h\times\bm V_h\times W_h\times W_h\times M_h(o)\times M_h(o)\times M_h(\partial) $.

Next, we present a basic property of the operators $\mathscr B_1$ and $\mathscr B_2$,   and show the HDG equations \eqref{HDG_full_discrete} have a unique solution.
\begin{lemma}\label{property_B}
	For any $ ( \bm v_h, w_h, \mu_h ) \in \bm V_h \times W_h \times M_h(o) $, we have
	\begin{align*}
	\hspace{2em}&\hspace{-2em} \mathscr B_1(\bm v_h,w_h,\mu_h;\bm v_h,w_h,\mu_h)\\
	&=(\bm v_h,\bm v_h)_{\mathcal T_h}+ \langle (h^{-1}+\tau_1 - \frac 12 \bm \beta\cdot\bm n)(w_h-\mu_h),w_h-\mu_h\rangle_{\partial\mathcal T_h\backslash \varepsilon_h^\partial}\\
	&\quad-\frac 1 2(\nabla\cdot\bm\beta w_h,w_h)_{\mathcal T_h} +\langle (h^{-1}+\tau_1-\frac12\bm \beta\cdot\bm n) w_h,w_h\rangle_{\varepsilon_h^\partial},\\
	\hspace{2em}&\hspace{-2em}\mathscr B_2(\bm v_h,w_h,\mu_h;\bm v_h,w_h,\mu_h)\\
	&=(\bm v_h,\bm v_h)_{\mathcal T_h}+ \langle (h^{-1}+\tau_2 + \frac 12 \bm \beta\cdot\bm n)(w_h-\mu_h),w_h-\mu_h\rangle_{\partial\mathcal T_h\backslash \varepsilon_h^\partial}\\
	&\quad-\frac 1 2(\nabla\cdot\bm\beta w_h,w_h)_{\mathcal T_h} +\langle (h^{-1}+\tau_2+\frac12\bm \beta\cdot\bm n) w_h,w_h\rangle_{\varepsilon_h^\partial}.
	\end{align*}
\end{lemma}
\begin{proof}
	We only prove the first identity; the second can be obtained by the same argument.
	\begin{align*}
	\hspace{3em}&\hspace{-3em} 	\mathscr B_1(\bm v_h,w_h,\mu_h;\bm v_h,w_h,\mu_h)\\
	&=(\bm v_h,\bm v_h)_{\mathcal T_h}-( w_h,\nabla\cdot\bm v_h)_{\mathcal T_h}+\langle \mu_h,\bm v_h\cdot\bm n\rangle_{\partial\mathcal T_h\backslash \varepsilon_h^\partial}\\
	& \quad -(\bm v_h+\bm \beta w_h,  \nabla w_h)_{\mathcal T_h}-(\nabla\cdot\bm\beta w_h,w_h)_{\mathcal T_h}\\
	& \quad +\langle {\bm v}_h\cdot\bm n +(h^{-1}+\tau_1)w_h,w_h\rangle_{\partial\mathcal T_h} +\langle (\bm\beta\cdot\bm n -h^{-1}-\tau_1) \mu_h,w_h\rangle_{\partial\mathcal T_h\backslash \varepsilon_h^\partial}\\
	& \quad-\langle  {\bm v}_h\cdot\bm n+\bm \beta\cdot\bm n\mu_h +(h^{-1}+\tau_1)(w_h - \mu_h),\mu_h \rangle_{\partial\mathcal T_h\backslash\varepsilon^{\partial}_h}\\
	&=(\bm v_h,\bm v_h)_{\mathcal T_h}-(\bm \beta w_h,  \nabla w_h)_{\mathcal T_h} -(\nabla\cdot\bm\beta w_h,w_h)_{\mathcal T_h}\\
	&\quad+\langle  (h^{-1} +\tau_1) w_h,w_h\rangle_{\partial\mathcal T_h}+\langle (\bm\beta\cdot\bm n -h^{-1}-\tau_1) \mu_h,w_h\rangle_{\partial\mathcal T_h\backslash \varepsilon_h^\partial}\\
	&\quad -\langle  \bm \beta\cdot\bm n \mu_h +(h^{-1}+ \tau_1)(w_h - \mu_h ),\mu_h\rangle_{\partial\mathcal T_h\backslash\varepsilon^{\partial}_h}.
	\end{align*}
	Moreover,
	\begin{align*}
	(\bm \beta w_h,\nabla w_h)_{\mathcal T_h}&=(\bm \beta\cdot\nabla w_h,w_h)_{\mathcal T_h}=(\nabla\cdot(\bm \beta w_h),w_h)_{\mathcal T_h}-(\nabla\cdot\bm \beta w_h,w_h)_{\mathcal T_h}\\
	&=\langle\bm \beta\cdot\bm n w_h,w_h\rangle_{\partial\mathcal T_h}-(\bm \beta w_h,\nabla w_h)_{\mathcal T_h}-(\nabla\cdot\bm \beta w_h,w_h)_{\mathcal T_h},
	\end{align*}
	which implies
	\begin{align*}
	(\bm \beta w_h,\nabla w_h)_{\mathcal T_h}=\frac12\langle\bm \beta\cdot\bm n w_h,w_h\rangle_{\partial\mathcal T_h}-\frac12(\nabla\cdot\bm \beta w_h,w_h)_{\mathcal T_h}.
	\end{align*}
	Then we obtain
	\begin{align*}
	\hspace{1em}&\hspace{-1em}  \mathscr B_1 (\bm v_h,w_h,\mu_h;\bm v_h,w_h,\mu_h)\\
	&=(\bm v_h,\bm v_h)_{\mathcal T_h}+ \langle (h^{-1} + \tau_1 - \frac 12 \bm \beta\cdot\bm n)(w_h-\mu_h),w_h-\mu_h\rangle_{\partial\mathcal T_h\backslash \varepsilon_h^\partial}\\
	&\quad  -\frac{1}{2}(\nabla\cdot\bm\beta w_h,w_h)_{\mathcal T_h} + \langle (h^{-1}+\tau_1-\frac12\bm \beta\cdot\bm n) w_h,w_h\rangle_{\varepsilon_h^\partial}-\frac12\langle\bm \beta\cdot\bm n \mu_h,\mu_h\rangle_{\partial\mathcal T_h\backslash \varepsilon_h^\partial}.
	\end{align*}
	Since $ \mu_h$ is single-valued across the interfaces, we have
	\begin{align*}
	-\frac12\langle\bm \beta\cdot\bm n\mu_h,\mu_h\rangle_{\partial\mathcal T_h\backslash\varepsilon_h^\partial}=0.
	\end{align*}
	This completes the proof.
\end{proof}

Next, we give a property of the HDG operators $\mathscr B_1$ and $\mathscr B_2$ that is critical to our error analysis of the method.
\begin{lemma}\label{identical_equa}
	If \textbf{(A2)} holds, then
	$$\mathscr B_1 (\bm q_h,y_h,\widehat y_h^o;\bm p_h,-z_h,-\widehat z_h^o) + \mathscr B_2 (\bm p_h,z_h,\widehat z_h^o;-\bm q_h,y_h,\widehat y_h^o) = 0.$$
\end{lemma}
\begin{proof}
	By the definition of $ \mathscr B_1 $  and $ \mathscr B_2 $,
	\begin{align*}
	\hspace{1em}&\hspace{-1em}  \mathscr B_1 (\bm q_h,y_h,\widehat y_h^o;\bm p_h,-z_h,-\widehat z_h^o) + \mathscr B_2 (\bm p_h,z_h,\widehat z_h^o;-\bm q_h,y_h,\widehat y_h^o)\\
	&=(\bm{q}_h, \bm p_h)_{{\mathcal{T}_h}}- (y_h, \nabla\cdot \bm p_h)_{{\mathcal{T}_h}}+\langle \widehat{y}_h^o, \bm p_h\cdot \bm{n} \rangle_{\partial{{\mathcal{T}_h}}\backslash {\varepsilon_h^{\partial}}} \\
	&\quad + (\bm{q}_h + \bm{\beta} y_h, \nabla z_h)_{{\mathcal{T}_h}} + (\nabla\cdot\bm \beta y_h,  z_h)_{{\mathcal{T}_h}}  - \langle\bm q_h\cdot\bm n + (h^{-1}+\tau_1) y_h , z_h \rangle_{\partial{{\mathcal{T}_h}}}\\
	&\quad  - \langle(\bm\beta\cdot \bm n-\tau_1-h^{-1})\widehat y_h^o, z_h \rangle_{\partial{{\mathcal{T}_h}}\backslash \varepsilon_h^{\partial}} \\
	&\quad+ \langle\bm q_h\cdot\bm n + \bm{\beta}\cdot\bm n \widehat y_h^o  +(h^{-1}+\tau_1)(y_h-\widehat y_h^o), \widehat z_h^o  \rangle_{\partial{{\mathcal{T}_h}}\backslash\varepsilon_h^{\partial}}\\
	&\quad-(\bm{p}_h, \bm q_h)_{{\mathcal{T}_h}}+ (z_h, \nabla\cdot \bm q_h)_{{\mathcal{T}_h}} -\langle \widehat{z}_h^o, \bm q_h \cdot \bm{n} \rangle_{\partial{{\mathcal{T}_h}}\backslash {\varepsilon_h^{\partial}}}  - (\bm{p}_h - \bm{\beta} z_h, \nabla y_h)_{{\mathcal{T}_h}} \\
	&\quad+\langle\bm p_h\cdot\bm n +(h^{-1}+\tau_2) z_h , y_h \rangle_{\partial{{\mathcal{T}_h}}} - \langle (\bm{\beta}\cdot \bm n+\tau_2 + h^{-1}) \widehat z_h^o, y_h \rangle_{\partial{{\mathcal{T}_h}}\backslash \varepsilon_h^{\partial}}\\
	&\quad- \langle\bm p_h\cdot\bm n -\bm{\beta} \cdot\bm n\widehat z_h^o +(h^{-1}+ \tau_2) (z_h-\widehat z_h^o), \widehat y_h^o \rangle_{\partial{{\mathcal{T}_h}}\backslash\varepsilon_h^{\partial}}.
	\end{align*}
	Integration by parts gives
	\begin{align*}
	\mathscr B_1 &(\bm q_h,y_h,\widehat y_h^o;\bm p_h,-z_h,-\widehat z_h^o) + \mathscr B_2 (\bm p_h,z_h,\widehat z_h^o;-\bm q_h,y_h,\widehat y_h^o)\\
	&=\langle (\tau_2 + \bm{\beta}\cdot\bm n-\tau_1) y_h, z_h \rangle_{\partial\mathcal T_h} + \langle (\tau_2 + \bm{\beta}\cdot\bm n-\tau_1) \widehat y_h^o, \widehat z_h^o \rangle_{\partial\mathcal T_h\backslash\varepsilon_h^\partial}.
	\end{align*}
	The proof is complete by assumption \textbf{(A2)}.
\end{proof}

\begin{prop}\label{ex_uni}
	If \textbf{(A2)} holds, there exists a unique solution of the HDG equations \eqref{HDG_full_discrete}.
\end{prop}
\begin{proof}
	Since the system \eqref{HDG_full_discrete} is finite dimensional, we only need to prove the uniqueness.  Therefore, we assume $y_d = f =0$ and we show the system \eqref{HDG_full_discrete} only has the trivial solution.
	
	First, take $(\bm r_1,w_1,\mu_1) = (\bm p_h,-z_h,-\widehat z_h^o)$, $(\bm r_2,w_2,\mu_2) = (-\bm q_h,y_h,\widehat y_h^o)$, and $w_3 = -\gamma u_h $ in the HDG equations \eqref{HDG_full_discrete_a},  \eqref{HDG_full_discrete_b}, and \eqref{HDG_full_discrete_e}, respectively, by  \Cref{identical_equa} and sum to obtain
	\begin{align*}
	\hspace{1em}&\hspace{-1em}  \mathscr B_1  (\bm q_h,y_h,\widehat y_h^o;\bm p_h,-z_h,-\widehat z_h^o) + \mathscr B_2 (\bm p_h,z_h,\widehat z_h^o;-\bm q_h,y_h,\widehat y_h^o) \\
	& =	(y_h,y_h)_{\mathcal T_h} + \gamma \norm {u_h}^2_{\varepsilon_h^\partial} \\
	& = 0.
	\end{align*}
	This implies $y_h =  u_h = 0$ since $\gamma>0$.
	
	Next, taking $(\bm r_1,w_1,\mu_1) = (\bm q_h,y_h,\widehat y_h^o)$ and $(\bm r_2,w_2,\mu_2) = (\bm p_h,z_h,\widehat z_h^o)$ in Lemma \ref{property_B} gives $\bm q_h= \bm p_h= \bm 0 $, $ \widehat y_h^o =0$,  $ z_h = 0$ on $\varepsilon_h^{\partial}$, and $ z_h - \widehat z_h^o=0$ on $\partial\mathcal T_h\backslash \varepsilon_h^{\partial}$.  Also, since $\widehat z_h=0$ on $\varepsilon_h^{\partial}$ we have
	\begin{align}\label{uniqu}
	z_h-\widehat z_h =0.
	\end{align}
	Substituting \eqref{uniqu} into \eqref{HDG_discrete2_c}, and remembering again $\widehat z_h = 0$ on $\varepsilon_h^\partial$, we get
	\begin{align*}
	- (z_h, \nabla\cdot \bm{r_2})_{{\mathcal{T}_h}}+\langle  z_h, \bm{r_2}\cdot \bm{n} \rangle_{\partial{{\mathcal{T}_h}}} =0.
	\end{align*}
	Integrate by parts, and take $\bm r_2 = \nabla z_h$ to obtain
	\begin{align*}
	(\nabla z_h, \nabla z_h)_{{\mathcal{T}_h}}=0.
	\end{align*}
	Thus, $z_h$ is  constant on each $K\in\mathcal T_h$, and also $ z_h = \widehat z_h $ on $\partial\mathcal T_h $.  Since $\widehat z_h=0$ on $\varepsilon_h^{\partial}$ and single valued on each face, we have $z_h=0$ on each $K\in\mathcal T_h$, and therefore also $ \widehat{z}_h^o = 0 $.
\end{proof}

\subsection{Proof of Main Result}
To prove the main result, we follow the strategy of our earlier work \cite{HuShenSinglerZhangZheng_HDG_Dirichlet_control1} and split the proof into
seven steps.  We consider the following auxiliary problem: find $$({\bm{q}}_h(u),{\bm{p}}_h(u), y_h(u), z_h(u), {\widehat{y}}_h^o(u), {\widehat{z}}_h^o(u))\in \bm{V}_h\times\bm{V}_h\times W_h \times W_h\times M_h(o)\times M_h(o)$$ such that
\begin{subequations}\label{HDG_inter_u}
	\begin{align}
	\mathscr B_1(\bm q_h(u), y_h(u), \widehat{y}_h(u); \bm r_1, w_1, \mu_1)&=( f, w_1)_{\mathcal T_h} - \langle P_M u, \bm r_1\cdot\bm n\rangle_{\varepsilon_h^\partial} \nonumber \\
	& \quad -\langle (\bm\beta\cdot\bm n-h^{-1}-\tau_1) P_M u, w_1\rangle_{\varepsilon_h^\partial},\label{HDG_u_a}\\
	\mathscr B_2(\bm p_h(u), z_h(u), \widehat{z}_h(u); \bm r_2,  w_2, \mu_2)&=(y_h(u) - y_d, w_2)_{\mathcal T_h},\label{HDG_u_b}
	\end{align}
\end{subequations}
for all $\left(\bm{r}_1, \bm{r}_2,w_1,w_2,\mu_1,\mu_2\right)\in \bm{V}_h\times\bm{V}_h \times W_h\times W_h\times M_h(o)\times M_h(o)$. We first bound the error between the solutions of the auxiliary problem and the mixed form \eqref{mixed_a} -  \eqref{mixed_d}  of the optimality system. We use the following notation:
\begin{equation}\label{notation_1}
\begin{split}
\delta^{\bm q} &=\bm q-{\bm\Pi}\bm q,  \qquad\qquad\qquad \qquad\qquad\qquad\;\;\;\;\varepsilon^{\bm q}_h={\bm\Pi} \bm q-\bm q_h(u),\\
\delta^y&=y- {\Pi} y, \qquad\qquad\qquad \qquad\qquad\qquad\;\;\;\; \;\varepsilon^{y}_h={\Pi} y-y_h(u),\\
\delta^{\widehat y} &= y-P_My,  \qquad\qquad\qquad\qquad\qquad\qquad \;\;\; \varepsilon^{\widehat y}_h=P_M y-\widehat{y}_h(u),\\
\widehat {\bm\delta}_1 &= \delta^{\bm q}\cdot\bm n + \bm{\beta}\cdot\bm n \delta^{\widehat y} + (h^{-1}+\tau_1)(\delta^y - \delta^{\widehat y}),
\end{split}
\end{equation}
where $\widehat y_h(u) = \widehat y_h^o(u)$ on $\varepsilon_h^o$ and $\widehat y_h(u) = P_M u$ on $\varepsilon_h^{\partial}$.  Note that this implies $\varepsilon_h^{\widehat y} = 0$ on $\varepsilon_h^{\partial}$.

\subsubsection{Step 1: The error equation for part 1 of the auxiliary problem \eqref{HDG_u_a}.}
\label{subsec:proof_step1}

\begin{lemma}\label{lemma:step1_first_lemma}
	We have
	\begin{align}\label{error_equation_L2k1}
	\mathscr B_1 (\varepsilon_h^{\bm q},\varepsilon_h^{ y}, \varepsilon_h^{\widehat y}, \bm r_1, w_1, \mu_1) &= ( \bm \beta \delta^y, \nabla w_1)_{{\mathcal{T}_h}}
	+(\nabla\cdot\bm{\beta}\delta^y,w_1)_{\mathcal T_h}\nonumber\\
	& \quad - \langle \widehat{\bm \delta}_1, w_1 \rangle_{\partial{{\mathcal{T}_h}}} + \langle \widehat{\bm \delta}_1, \mu_1 \rangle_{\partial{{\mathcal{T}_h}}\backslash \varepsilon_h^{\partial}}.
	\end{align}
\end{lemma}
\begin{proof}
	By the definition of the operator $ \mathscr B_1 $ in \eqref{def_B1}, we have
	\begin{align*}
	\hspace{1em}&\hspace{-1em}  \mathscr B_1 (\bm \Pi {\bm q},\Pi { y}, P_M  y, \bm r_1, w_1, \mu_1)\\
	&= (\bm \Pi {\bm q}, \bm{r_1})_{{\mathcal{T}_h}}- (\Pi { y}, \nabla\cdot \bm{r_1})_{{\mathcal{T}_h}}+\langle P_M  y, \bm{r_1}\cdot \bm{n} \rangle_{\partial{{\mathcal{T}_h}}\backslash {\varepsilon_h^{\partial}}}\\
	&  \quad- (\bm \Pi {\bm q} + \bm{\beta} \Pi y, \nabla w_1)_{{\mathcal{T}_h}}
	- (\nabla\cdot\bm{\beta} \Pi y,  w_1)_{{\mathcal{T}_h}}\\
	&\quad+\langle \bm \Pi {\bm q}\cdot\bm n +(h^{-1} + \tau_1) \Pi { y}, w_1 \rangle_{\partial{{\mathcal{T}_h}}}
	+ (\bm{\beta}\cdot\bm n-h^{-1}-\tau_1) P_M  y, w_1 \rangle_{\partial{{\mathcal{T}_h}}\backslash \varepsilon_h^{\partial}}\\
	&\quad	-\langle  \bm \Pi \bm q\cdot\bm n+\bm \beta\cdot\bm n P_M y +(h^{-1}+\tau_1)(\Pi y - P_M y),\mu_1\rangle_{\partial\mathcal T_h\backslash\varepsilon^{\partial}_h}.
	\end{align*}
	By properties of the $ L^2 $ projections \eqref{L2_projection}, we have
	\begin{align*}
	\hspace{3em}&\hspace{-3em} \mathscr B_1 (\bm \Pi {\bm q},\Pi { y},P_M  y, \bm r_1, w_1, \mu_1) \\
	&= ( {\bm q}, \bm{r_1})_{{\mathcal{T}_h}}- ({ y}, \nabla\cdot \bm{r_1})_{{\mathcal{T}_h}}+\langle   y, \bm{r_1}\cdot \bm{n} \rangle_{\partial{{\mathcal{T}_h}}\backslash {\varepsilon_h^{\partial}}}\\
	& \quad - ( {\bm q} + \bm \beta y, \nabla w_1)_{{\mathcal{T}_h}} +  (  \bm \beta \delta^y, \nabla w_1)_{{\mathcal{T}_h}} - (\nabla\cdot\bm{\beta} y, w_1)_{\mathcal T_h} + (\nabla\cdot\bm{\beta} \delta^y, w_1)_{\mathcal T_h} \\
	&\quad+\langle {\bm q}\cdot\bm n, w_1 \rangle_{\partial{{\mathcal{T}_h}}} - \langle \delta^{\bm q}\cdot\bm n, w_1 \rangle_{\partial{{\mathcal{T}_h}}}+\langle (h^{-1}+\tau_1)\Pi y , w_1 \rangle_{\partial{{\mathcal{T}_h}}}  \\
	&\quad+\langle\bm\beta\cdot\bm n y, w_1\rangle_{\partial\mathcal T_h\backslash\varepsilon_h^\partial} - \langle\bm\beta\cdot\bm n \delta^{\widehat y}, w_1\rangle_{\partial\mathcal T_h\backslash\varepsilon_h^\partial}
	- \langle (h^{-1}+\tau_1) P_M  y, w_1 \rangle_{\partial{{\mathcal{T}_h}}\backslash \varepsilon_h^{\partial}} \\
	&\quad- \langle {\bm q}\cdot\bm n, \mu_1 \rangle_{\partial{{\mathcal{T}_h}}\backslash\varepsilon_h^{\partial}} + \langle \delta^{\bm q}\cdot\bm n, \mu_1 \rangle_{\partial{{\mathcal{T}_h}}\backslash\varepsilon_h^{\partial}} - \langle {\bm \beta}\cdot\bm n y, \mu_1 \rangle_{\partial{{\mathcal{T}_h}}\backslash\varepsilon_h^{\partial}}\\
	&\quad+\langle {\bm \beta}\cdot\bm n \delta^{\widehat y}, \mu_1 \rangle_{\partial{{\mathcal{T}_h}}\backslash\varepsilon_h^{\partial}}+\langle (h^{-1}+\tau_1) (\delta^y-\delta^{\widehat y}), \mu_1 \rangle_{\partial{{\mathcal{T}_h}}\backslash\varepsilon_h^{\partial}}.
	\end{align*}
	Note that the exact state $ y $ and exact flux $\bm{q}$ satisfy
	\begin{align*}
	(\bm{q},\bm{r}_1)_{\mathcal{T}_h}-(y,\nabla\cdot \bm{r}_1)_{\mathcal{T}_h}+\left\langle{y},\bm r_1\cdot \bm n \right\rangle_{\partial {\mathcal{T}_h}} &= 0,\\
	-(\bm{q}+\bm{\beta} y,\nabla w_1)_{\mathcal{T}_h}-(\nabla \cdot \bm{\beta} y, w_1)_{\mathcal{T}_h}+\left\langle ({\bm{q}}+\bm \beta y)\cdot \bm{n},w_1\right\rangle_{\partial {\mathcal{T}_h}} &= (f,w_1)_{\mathcal{T}_h},\\
	\left\langle ({\bm{q}}+\bm{\beta} y)\cdot \bm{n},\mu_1\right\rangle_{\partial {\mathcal{T}_h}\backslash \varepsilon_h^{\partial}}&=0,
	\end{align*}
	for all $(\bm{r}_1,w_1,\mu_1)\in\bm{V}_h\times W_h\times M_h(o)$. Then we have
	\begin{align*}
	\hspace{1em}&\hspace{-1em}  \mathscr B_1 (\bm \Pi {\bm q},\Pi { y}, P_M  y, \bm r_1, w_1, \mu_1)\\
	&=-\left\langle u,\bm r_1\cdot \bm n \right\rangle_{\varepsilon_h^{\partial}} - \left\langle \bm{\beta}\cdot \bm n u,w_1\right\rangle_{\varepsilon_h^{\partial}} + (f,w_1)_{\mathcal T_h} +  ( \bm \beta \delta^y, \nabla w_1)_{{\mathcal{T}_h}} \\
	&\quad+(\nabla\cdot\bm{\beta}\delta^y,w_1)_{\mathcal T_h} - \langle \delta^{\bm q}\cdot\bm n, w_1 \rangle_{\partial{{\mathcal{T}_h}}}+\langle (h^{-1} +\tau_1) \Pi y, w_1 \rangle_{\partial{{\mathcal{T}_h}}} \\
	&\quad- \langle {\bm \beta}\cdot\bm n \delta^{\widehat y}, w_1 \rangle_{\partial{{\mathcal{T}_h}}\backslash\varepsilon_h^{\partial}}  - \langle (h^{-1}+\tau_1) P_M  y, w_1 \rangle_{\partial{{\mathcal{T}_h}}\backslash \varepsilon_h^{\partial}} + \langle \delta^{\bm q}\cdot\bm n, \mu_1 \rangle_{\partial{{\mathcal{T}_h}}\backslash\varepsilon_h^{\partial}}\\
	&\quad+\langle {\bm \beta}\cdot\bm n \delta^{\widehat y}, \mu_1 \rangle_{\partial{{\mathcal{T}_h}}\backslash\varepsilon_h^{\partial}}+ \langle (h^{-1}+\tau_1) (\delta^y-\delta^{\widehat y}), \mu_1 \rangle_{\partial{{\mathcal{T}_h}}\backslash\varepsilon_h^{\partial}}.
	\end{align*}
	Subtract part 1 of the auxiliary problem \eqref{HDG_u_a} from the above equality to obtain the result:
	\begin{align*}
	\hspace{1em}&\hspace{-1em} \mathscr B_1 (\varepsilon_h^{\bm q},\varepsilon_h^{ y}, \varepsilon_h^{\widehat y},\bm r_1, w_1, \mu_1) \\
	& =  ( \bm \beta \delta^y, \nabla w_1)_{{\mathcal{T}_h}}
	+(\nabla\cdot\bm{\beta}\delta^y,w_1)_{\mathcal T_h} - \langle \delta^{\bm q}\cdot\bm n, w_1 \rangle_{\partial{{\mathcal{T}_h}}} \\
	&  \quad+\langle (h^{-1}+\tau_1) \Pi y, w_1 \rangle_{\partial{{\mathcal{T}_h}}}- \langle {\bm \beta}\cdot\bm n \delta^{\widehat y}, w_1 \rangle_{\partial{{\mathcal{T}_h}}}  - \langle (h^{-1}+\tau_1) P_M  y, w_1 \rangle_{\partial{{\mathcal{T}_h}}} \\
	&\quad+ \langle \delta^{\bm q}\cdot\bm n, \mu_1 \rangle_{\partial{{\mathcal{T}_h}}\backslash\varepsilon_h^{\partial}}+\langle {\bm \beta}\cdot\bm n \delta^{\widehat y}, \mu_1 \rangle_{\partial{{\mathcal{T}_h}}\backslash\varepsilon_h^{\partial}}+ \langle (h^{-1}+\tau_1) (\delta^y-\delta^{\widehat y}), \mu_1 \rangle_{\partial{{\mathcal{T}_h}}\backslash\varepsilon_h^{\partial}}\\
	&= ( \bm \beta \delta^y, \nabla w_1)_{{\mathcal{T}_h}}
	+(\nabla\cdot\bm{\beta}\delta^y,w_1)_{\mathcal T_h} - \langle \widehat{\bm \delta}_1, w_1 \rangle_{\partial{{\mathcal{T}_h}}} + \langle \widehat{\bm \delta}_1, \mu_1 \rangle_{\partial{{\mathcal{T}_h}}\backslash \varepsilon_h^{\partial}}.
	\end{align*}
\end{proof}

\subsubsection{Step 2: Estimate for $\varepsilon_h^{\boldmath q}$.}
We begin with a key inequality that has been proved for an existing HDG method in \cite{MR3440284}.
\begin{lemma}\label{nabla_ine}
	We have
	\begin{equation*}
	\|\nabla\varepsilon^y_h\|_{\mathcal T_h}
	\le \|\varepsilon^{\bm q}_h\|_{\mathcal T_h}+Ch^{-\frac1 2}\|\varepsilon^y_h-\varepsilon^{\widehat y}_h\|_{\partial\mathcal T_h}.
	\end{equation*}
\end{lemma}

In the HDG method in \cite{MR3440284}, degree $ k $ polynomials are used for the space $ M_h $ instead of degree $ k+1 $ here.  Increasing this degree does not lead to any change in the proof of the above lemma; therefore, we omit the proof.


\begin{lemma}\label{lemma:step2_main_lemma}
	We have
	\begin{align}
	\norm{\varepsilon_h^{\bm{q}}}_{\mathcal{T}_h}^2+h^{-1}\|{\varepsilon_h^y-\varepsilon_h^{\widehat{y}}}\|_{\partial \mathcal T_h}^2 \lesssim h^{2s_{\bm q}}\norm{\bm q}_{s^{\bm q},\Omega}^2 + h^{2s_{y}-2}\norm{y}_{s^{y},\Omega}^2.
	\end{align}
\end{lemma}

\begin{proof}
	First, since $\varepsilon_h^{\widehat y}=0$ on $\varepsilon_h^\partial$, the basic property of $ \mathscr B_1 $ in \Cref{property_B} gives
	\begin{align*}
	\hspace{1em}&\hspace{-1em} \mathscr B(\varepsilon_h^{\bm q},\varepsilon_h^{ y}, \varepsilon_h^{\widehat y}, \varepsilon_h^{\bm q},\varepsilon_h^{ y}, \varepsilon_h^{\widehat y})\\ &=(\varepsilon_h^{\bm{q}},\varepsilon_h^{\bm{q}})_{\mathcal{T}_h}+ \|(h^{-1}+\tau_1-\frac 12 \bm{\beta} \cdot\bm n)^{\frac 12 } (\varepsilon_h^y-\varepsilon_h^{\widehat{y}})\|_{\partial\mathcal T_h}^2 +\frac 1 2\| (-\nabla\cdot\bm{\beta})^{\frac  1 2}\varepsilon_h^y\|_{\mathcal T_h}^2.
	\end{align*}
	Then, taking $(\bm r_1, w_1,\mu_1) = (\bm \varepsilon_h^{\bm q},\varepsilon_h^y,\varepsilon_h^{\widehat y})$ in \eqref{error_equation_L2k1} in \Cref{lemma:step1_first_lemma} gives
	\begin{equation}\label{step2_2}
	\begin{split}
	(\varepsilon_h^{\bm{q}},\varepsilon_h^{\bm{q}})_{\mathcal{T}_h}&+ \|(h^{-1}+\tau_1-\frac 12 \bm{\beta} \cdot\bm n)^{\frac 12 } (\varepsilon_h^y-\varepsilon_h^{\widehat{y}})\|_{\partial\mathcal T_h}^2 +\frac 1 2\| (-\nabla\cdot\bm{\beta})^{\frac  1 2}\varepsilon_h^y\|_{\mathcal T_h}^2\\
	&= ( \bm \beta \delta^y, \nabla \varepsilon_h^y)_{{\mathcal{T}_h}}
	+(\nabla\cdot\bm{\beta}\delta^y,\varepsilon_h^y)_{\mathcal T_h} -\langle \widehat {\bm\delta}_1,\varepsilon_h^y - \varepsilon_h^{\widehat y}\rangle_{\partial\mathcal T_h} \\
	&=: T_1 + T_2 + T_3.
	\end{split}
	\end{equation}
	For the terms $T_1$ and $T_2$, simply applying \Cref{nabla_ine} and Young's inequality gives
	\begin{align*}
	T_1 &=  ( \bm \beta \delta^y, \nabla \varepsilon_h^y)_{{\mathcal{T}_h}} \le C \| \delta^y\|_{\mathcal T_h}^2 + \frac 1 4
	\|\varepsilon_h^{\bm{q}}\|_{\mathcal T_h}^2 + \frac 1 {4h} \|{\varepsilon_h^y-\varepsilon_h^{\widehat{y}}}\|_{\partial \mathcal T_h}^2,\\
	T_2 &= (\nabla\cdot\bm{\beta}\delta^y,\varepsilon_h^y)_{\mathcal T_h} \le C\|\delta^y\|_{\mathcal T_h}^2 + \frac 1 2 \|(-\nabla\cdot\bm{\beta})^{\frac 1 2} \varepsilon_h^y\|_{\mathcal T_h}^2,\\
	T_3 &= - \langle \widehat {\bm\delta}_1,\varepsilon_h^y - \varepsilon_h^{\widehat y}\rangle_{\partial\mathcal T_h} \le 4h \|\bm{\delta}_1\|_{\partial\mathcal T_h}^2 + \frac 1 {4h} \|{\varepsilon_h^y-\varepsilon_h^{\widehat{y}}}\|_{\partial \mathcal T_h}^2.
	\end{align*}
	Sum all the estimates for $\{T_i\}_{i=1}^3$ to obtain
	\begin{align*}
	\|\varepsilon_h^{\bm{q}}\|_{\mathcal{T}_h}^2+h^{-1}\|{\varepsilon_h^y-\varepsilon_h^{\widehat{y}}}\|_{\partial \mathcal T_h}^2 & \lesssim h \|\bm{\delta}_1\|_{\partial\mathcal T_h}^2  +
	\norm{\delta^{y}}_{\mathcal T_h}^2 \\
	&\lesssim h^{2s_{\bm q}}\norm{\bm q}_{s^{\bm q},\Omega}^2 + h^{2s_{y}-2}\norm{y}_{s^{y},\Omega}^2.
	\end{align*}
\end{proof}

\subsubsection{Step 3: Estimate for $\varepsilon_h^{y}$ by a duality argument.}
\label{subsec:proof_step3}

Next, we introduce the dual problem for any given $\Theta$ in $L^2(\Omega):$
\begin{equation}\label{Dual_PDE}
\begin{split}
\bm\Phi-\nabla\Psi &= 0\qquad \ \ ~\text{in}\ \  \Omega,\\
\nabla\cdot\bm{\Phi}+\nabla\cdot(\bm\beta\Psi) &= \Theta \qquad \ \text{in}\  \ \Omega,\\
\Psi &= 0\qquad \ \ ~\text{on}\ \partial\Omega.
\end{split}
\end{equation}
Since the domain $\Omega$ is convex, we have the regularity estimate
\begin{align}
\norm{\bm \Phi}_{1,\Omega} + \norm{\Psi}_{2,\Omega} \le C_{\text{reg}} \norm{\Theta}_\Omega.
\end{align}

Before we estimate  $\varepsilon_h^y$,  we introduce the following notation, which is similar to the earlier notation in \eqref{notation_1}:
\begin{align}
\delta^{\bm \Phi} &=\bm \Phi-{\bm\Pi} \bm \Phi, \quad \delta^\Psi=\Psi- {\Pi} \Psi, \quad
\delta^{\widehat \Psi} = \Psi-P_M\Psi.
\end{align}
\begin{lemma}\label{e_sec}
	We have
	\begin{align*}
	\|\varepsilon_h^y\|_{\mathcal T_h} \lesssim  h^{s_{\bm q}+1}\norm{\bm q}_{s^{\bm q},\Omega} + h^{s_{y}}\norm{y}_{s^{y},\Omega}.
	\end{align*}
\end{lemma}

\begin{proof}
	Consider the dual problem \eqref{Dual_PDE} and let $\Theta =- \varepsilon_h^y$.  Take  $(\bm r_1,w_1,\mu_1) = ( {\bm\Pi}\bm{\Phi},{\Pi}\Psi,P_M\Psi)$ in \eqref{error_equation_L2k1} in \Cref{lemma:step1_first_lemma},   and since $\Psi = 0$ on $\varepsilon_h^{\partial}$, we have
	\begin{align*}
	\hspace{1em}&\hspace{-1em}  \mathscr B_1 (\varepsilon^{\bm q}_h,\varepsilon^y_h,\varepsilon^{\widehat y}_h;{\bm\Pi}\bm{\Phi},{\Pi}\Psi,P_M\Psi)\\
	&= (\varepsilon^{\bm q}_h,{\bm\Pi}\bm{\Phi})_{\mathcal T_h}-( \varepsilon^y_h,\nabla\cdot{\bm\Pi}\bm{\Phi})_{\mathcal T_h}+\langle  \varepsilon^{\widehat y}_h,{\bm\Pi}\bm{\Phi}\cdot\bm n\rangle_{\partial\mathcal T_h\backslash \varepsilon_h^\partial}\\
	&\quad-(\varepsilon^{\bm q}_h+\bm \beta\varepsilon^y_h,  \nabla {\Pi}\Psi)_{\mathcal T_h}-(\nabla\cdot\bm\beta \varepsilon^y_h,{\Pi}\Psi)_{\mathcal T_h}+\langle \varepsilon^{\bm q}_h\cdot\bm n +(h^{-1}+\tau_1) \varepsilon^y_h ,{\Pi}\Psi\rangle_{\partial\mathcal T_h}\\
	&\quad+\langle (\bm\beta\cdot\bm n -h^{-1}-\tau_1) \varepsilon^{\widehat y}_h,{\Pi}\Psi\rangle_{\partial\mathcal T_h}\\
	&\quad-\langle  \varepsilon^{\bm q}_h\cdot\bm n+\bm \beta\cdot\bm n\varepsilon^{\widehat y}_h  + (h^{-1}+\tau_1)(\varepsilon^y_h - \varepsilon^{\widehat y}_h),P_M\Psi\rangle_{\partial\mathcal T_h}\\
	&= (\varepsilon^{\bm q}_h,\bm{\Phi})_{\mathcal T_h}-( \varepsilon^y_h,\nabla\cdot \bm{\Phi})_{\mathcal T_h} + ( \varepsilon^y_h,\nabla\cdot\delta^{\bm{\Phi}})_{\mathcal T_h}-\langle  \varepsilon^{\widehat y}_h,\delta^{\bm{\Phi}}\cdot\bm n\rangle_{\partial\mathcal T_h}-(\varepsilon^{\bm q}_h+\bm \beta\varepsilon^y_h,  \nabla \Psi)_{\mathcal T_h}\\
	&\quad+(\varepsilon^{\bm q}_h+\bm \beta\varepsilon^y_h,  \nabla \delta^\Psi)_{\mathcal T_h} - (\nabla\cdot\bm\beta \varepsilon^y_h,\Psi)_{\mathcal T_h} + (\nabla\cdot\bm\beta \varepsilon^y_h,\delta^\Psi)_{\mathcal T_h}\\
	& \quad- \langle  \varepsilon^{\bm q}_h\cdot\bm n+\bm \beta\cdot\bm n\varepsilon^{\widehat y}_h +(h^{-1}+ \tau_1)(\varepsilon^y_h - \varepsilon^{\widehat y}_h),\delta^\Psi - \delta^{\widehat\Psi}\rangle_{\partial\mathcal T_h}.
	\end{align*}
	Here we used $\langle\varepsilon^{\widehat y}_h,\bm \Phi\cdot\bm n\rangle_{\partial\mathcal T_h}=0$, which holds since $\varepsilon^{\widehat y}_h$ is single-valued function on interior edges and $\varepsilon^{\widehat y}_h=0$ on $\varepsilon^{\partial}_h$.
	
	Next, integration by parts gives
	\begin{equation}\label{inde_eq2}
	\begin{split}
	(\varepsilon^y_h,\nabla\cdot\delta^{\bm \Phi})_{\mathcal{T}_h}
	&= \langle \varepsilon^y_h,\delta^{\bm \Phi} \cdot\bm n\rangle_{\partial\mathcal T_h}-(\nabla\varepsilon^y_h,\delta^{\bm \Phi})_{\mathcal{T}_h} = \langle \varepsilon^y_h,\delta^{\bm \Phi}\cdot\bm n\rangle_{\partial\mathcal T_h},\\
	(\varepsilon^{\bm q}_h, \nabla \delta^{ \Psi})_{\mathcal{T}_h}&=\langle \varepsilon^{\bm q}_h \cdot\bm n, \delta^{ \Psi}\rangle_{\partial\mathcal T_h}-(\nabla\cdot \varepsilon^{\bm q}_h , \delta^{ \Psi})_{\mathcal T_h} = \langle \varepsilon^{\bm q}_h \cdot\bm n, \delta^{ \Psi}\rangle_{\partial\mathcal T_h},\\
	(\bm\beta \varepsilon_h^y, \nabla \delta^{ \Psi})_{\mathcal{T}_h}&=\langle \bm{\beta} \cdot\bm n \varepsilon_h^y, \delta^{ \Psi}\rangle_{\partial\mathcal T_h}- (\nabla\cdot \bm{\beta} \varepsilon_h^y, \delta^{ \Psi})_{\mathcal T_h}  - (\bm{\beta} \nabla\varepsilon_h^y, \delta^{ \Psi})_{\mathcal T_h}.
	\end{split}
	\end{equation}
	We have
	\begin{align*}
	\hspace{3em}&\hspace{-3em}  \mathscr B_1 (\varepsilon^{\bm q}_h,\varepsilon^y_h,\varepsilon^{\widehat y}_h;{\bm\Pi}\bm{\Phi},{\Pi}\Psi,P_M\Psi)\\
	& =\|  \varepsilon_h^y\|_{\mathcal T_h}^2 + \langle \varepsilon^y_h - \varepsilon^{\widehat y}_h,\delta^{\bm \Phi}\cdot\bm n +\bm{\beta}\cdot\bm n  \delta^{\Psi} \rangle_{\partial\mathcal T_h} - (\nabla \varepsilon_h^y,\bm{\beta}\delta^{\Psi})_{\mathcal T_h}\\
	& \quad - \langle   (h^{-1}+ \tau_1)(\varepsilon^y_h-\varepsilon^{\widehat y}_h),\delta^\Psi - \delta^{\widehat\Psi}\rangle_{\partial\mathcal T_h}.
	\end{align*}
	On the other hand,  $\Psi = 0$ on $\varepsilon_h^\partial $  and \eqref{error_equation_L2k1} in \Cref{lemma:step1_first_lemma} give
	\begin{align*}
	\hspace{3em}&\hspace{-3em}  \mathscr B_1 (\varepsilon^{\bm q}_h,\varepsilon^y_h,\varepsilon^{\widehat y}_h;{\bm\Pi}\bm{\Phi},{\Pi}\Psi,P_M\Psi)\\
	&= ( \bm \beta \delta^y, \nabla {\Pi}\Psi)_{{\mathcal{T}_h}}
	+(\nabla\cdot\bm{\beta}\delta^y,{\Pi}\Psi)_{\mathcal T_h} + \langle \widehat{\bm \delta}_1,\delta^{\Psi} - \delta^{\widehat \Psi} \rangle_{\partial{{\mathcal{T}_h}}}.
	\end{align*}
	Comparing the above two equalities, we get
	\begin{align*}
	\|  \varepsilon_h^y\|_{\mathcal T_h}^2  &= - \langle \varepsilon^y_h - \varepsilon^{\widehat y}_h,\delta^{\bm \Phi}\cdot\bm n +\bm{\beta}\cdot\bm n  \delta^{\Psi} \rangle_{\partial\mathcal T_h} \\
	& \quad + (\nabla \varepsilon_h^y,\bm{\beta}\delta^{\Psi})_{{\mathcal{T}_h}}+( \bm \beta \delta^y, \nabla {\Pi}\Psi)_{{\mathcal{T}_h}}+(\nabla\cdot\bm{\beta}\delta^y,{\Pi}\Psi)_{\mathcal T_h}\\
	&\quad +\langle   (h^{-1} + \tau_1)(\varepsilon^y_h-\varepsilon^{\widehat y}_h) 
	 \widehat{\bm \delta}_1,\delta^\Psi - \delta^{\widehat\Psi}\rangle_{\partial\mathcal T_h}\\
	&=:R_1+R_2+R_3+R_4+R_5.
	\end{align*}
	For the terms $R_1$ and $R_2$,  \Cref{nabla_ine} and \Cref{lemma:step2_main_lemma} give
	\begin{align*}
	R_1 &= - \langle \varepsilon^y_h - \varepsilon^{\widehat y}_h,\delta^{\bm \Phi}\cdot\bm n +\bm{\beta}\cdot\bm n  \delta^{\Psi} \rangle_{\partial\mathcal T_h}\\
	     &\le    h^{-\frac 1 2}\|\varepsilon^y_h - \varepsilon^{\widehat y}_h\|_{\partial \mathcal T_h} ~h^{\frac 1 2} \|\delta^{\bm \Phi}\cdot\bm n +\bm{\beta}\cdot\bm n  \delta^{\Psi}\|_{\partial\mathcal T_h}\\
	& \le   h^{-\frac 1 2}\|\varepsilon^y_h - \varepsilon^{\widehat y}_h\|_{\partial \mathcal T_h}  \|\delta^{\bm \Phi}\cdot\bm n +\bm{\beta}\cdot\bm n  \delta^{\Psi}\|_{\mathcal T_h}\\
	& \le  C h^{-\frac 1 2}\|\varepsilon^y_h - \varepsilon^{\widehat y}_h\|_{\partial \mathcal T_h}  (\|\delta^{\bm \Phi}\|_{\mathcal T_h}+ \|  \delta^{\Psi}\|_{\mathcal T_h})\\
	&\le C( h^{s_{\bm q}+1}\norm{\bm q}_{s^{\bm q},\Omega} + h^{s_{y}}\norm{y}_{s^{y},\Omega}) \|\varepsilon_h^y\|_{\mathcal T_h},\\
	\bigskip
	R_2 &=(\nabla \varepsilon_h^y,\bm{\beta}\delta^{\Psi})_{{\mathcal{T}_h}} \le C \|\nabla \varepsilon_h^y\|_{{\mathcal{T}_h}} \|\delta^{\Psi}\|_{{\mathcal{T}_h}}\\
	& \le C( h^{s_{\bm q}+1}\norm{\bm q}_{s^{\bm q},\Omega} + h^{s_{y}}\norm{y}_{s^{y},\Omega}) \|\varepsilon_h^y\|_{\mathcal T_h}.
	\end{align*}
	By a simple triangle inequality for terms $R_3$ and $R_4$, we have
	\begin{align*}
	R_3 &=( \bm \beta \delta^y, \nabla {\Pi}\Psi)_{{\mathcal{T}_h}}\le  C\|\delta^y\|_{\mathcal T_h} \|\nabla {\Pi}\Psi\|_{\mathcal T_h}\\
	&\le C\|\delta^y\|_{\mathcal T_h} ( \|\nabla \delta^\Psi\|_{\mathcal T_h} +\|\nabla\Psi\|_{\mathcal T_h})\\
	&\le C\|\delta^y\|_{\mathcal T_h} (h\|\Psi\|_{2,\Omega} + \|\Psi\|_{1,\Omega}  ) \le C\|\delta^y\|_{\mathcal T_h} \|\Psi\|_{2,\Omega}\\
	&\le C( h^{s_{\bm q}+1}\norm{\bm q}_{s^{\bm q},\Omega} + h^{s_{y}}\norm{y}_{s^{y},\Omega}) \|\varepsilon_h^y\|_{\mathcal T_h},\\
	R_4 &=(\nabla\cdot\bm{\beta}\delta^y,{\Pi}\Psi)_{\mathcal T_h}\le  C\|\delta^y\|_{\mathcal T_h} \| {\Pi}\Psi\|_{\mathcal T_h}\\
	&\le C\|\delta^y\|_{\mathcal T_h}  ( \| \delta^\Psi\|_{\mathcal T_h} +\|\Psi\|_{\mathcal T_h})\\
	&\le C\|\delta^y\|_{\mathcal T_h} (h^2\|\Psi\|_{2,\Omega} +\|\Psi\|_{\Omega}  ) \le C\|\delta^y\|_{\mathcal T_h} \|\Psi\|_{2,\Omega}\\
	&\le C( h^{s_{\bm q}+1}\norm{\bm q}_{s^{\bm q},\Omega} + h^{s_{y}}\norm{y}_{s^{y},\Omega}) \|\varepsilon_h^y\|_{\mathcal T_h}.
	\end{align*}
	For the term $R_5$, we have
	\begin{align*}
	R_5 &=\langle   (h^{-1}+ \tau_1)(\varepsilon^y_h-\varepsilon^{\widehat y}_h) + \widehat{\bm \delta}_1,\delta^\Psi - \delta^{\widehat\Psi}\rangle_{\partial\mathcal T_h}\\
	&\le C(h^{-1}\|(\varepsilon^y_h-\varepsilon^{\widehat y}_h) \|_{\partial\mathcal T_h} + \|\widehat{\bm \delta}_1\|_{\partial\mathcal T_h} )\|\delta^\Psi - \delta^{\widehat\Psi}\|_{\partial\mathcal T_h}\\
	&\le C( h^{s_{\bm q}+1}\norm{\bm q}_{s^{\bm q},\Omega} + h^{s_{y}}\norm{y}_{s^{y},\Omega}) \|\varepsilon_h^y\|_{\mathcal T_h}.
	\end{align*}
	Finally, we complete the proof  by summing the estimates for $R_1$ to $R_5$.
\end{proof}
As a consequence, a simple application of the triangle inequality gives optimal convergence rates for $\|\bm q -\bm q_h(u)\|_{\mathcal T_h}$ and $\|y -y_h(u)\|_{\mathcal T_h}$:

\begin{lemma}\label{lemma:step3_conv_rates}
	\begin{subequations}
		\begin{align}
		\|\bm q -\bm q_h(u)\|_{\mathcal T_h}&\le \|\delta^{\bm q}\|_{\mathcal T_h} + \|\varepsilon_h^{\bm q}\|_{\mathcal T_h} \lesssim  h^{s_{\bm q}}\norm{\bm q}_{s^{\bm q},\Omega} + h^{s_{y}-1}\norm{y}_{s^{y},\Omega},\\
		\|y -y_h(u)\|_{\mathcal T_h}&\le \|\delta^{y}\|_{\mathcal T_h} + \|\varepsilon_h^{y}\|_{\mathcal T_h} \lesssim h^{s_{\bm q}+1}\norm{\bm q}_{s^{\bm q},\Omega} + h^{s_{y}}\norm{y}_{s^{y},\Omega}.
		\end{align}
	\end{subequations}
\end{lemma}

\subsubsection{Step 4: The error equation for part 2 of the auxiliary problem \eqref{HDG_u_b}.}
Next, we focus on the dual variables, i.e., the state $z$ and the flux $\bm p$, and estimate the error between the solutions of the auxiliary problem and the mixed form \eqref{mixed_a} -  \eqref{mixed_d} of the optimality system. Define
\begin{equation}\label{notation_3}
\begin{split}
\delta^{\bm p} &=\bm p-{\bm\Pi}\bm p,  \qquad\qquad\qquad \qquad\qquad\qquad\;\;\;\;\varepsilon^{\bm p}_h={\bm\Pi} \bm p-\bm p_h(u),\\
\delta^z&=z- {\Pi} z, \qquad\qquad\qquad \qquad\qquad\qquad\;\;\;\; \;\varepsilon^{z}_h={\Pi} z-z_h(u),\\
\delta^{\widehat z} &= z-P_Mz,  \qquad\qquad\qquad\qquad\qquad\qquad \;\;\; \varepsilon^{\widehat z}_h=P_M z-\widehat{z}_h(u),\\
\widehat {\bm\delta}_2 &= \delta^{\bm p}\cdot\bm n + \bm{\beta}\cdot\bm n \delta^{\widehat z} + (h^{-1}+\tau_2)(\delta^z - \delta^{\widehat z}).
\end{split}
\end{equation}

\begin{lemma}\label{lemma:step4_first_lemma}
	We have
	\begin{align}
	\hspace{3em}&\hspace{-3em}  \mathscr B_2 (\varepsilon_h^{\bm p},\varepsilon_h^{z}, \varepsilon_h^{\widehat z}, \bm r_2, w_2, \mu_2) \nonumber\\
	&=( \bm \beta \delta^z, \nabla w_2)_{{\mathcal{T}_h}} - \langle \widehat{\bm \delta}_2, w_2 \rangle_{\partial{{\mathcal{T}_h}}} + \langle \widehat{\bm \delta}_2, \mu_2 \rangle_{\partial{{\mathcal{T}_h}}\backslash \varepsilon_h^{\partial}} +(y-y_h(u),w_2)_{\mathcal T_h}.\label{error_equation_L2k2}
	\end{align}
\end{lemma}
The proof is similar to the proof of \Cref{lemma:step4_first_lemma} and is omitted.

\subsubsection{Step 5: Estimate for $\varepsilon_h^{\boldmath p}$.}
Before we estimate  $\varepsilon_h^{\bm p}$, we give the following discrete Poincar{\'e} inequality from \cite{MR3440284}.
\begin{lemma}\label{lemma:discr_Poincare_ineq}
	We have
	\begin{align}\label{poin_in}
	\|\varepsilon_h^z\|_{\mathcal T_h} \le C(\|\nabla \varepsilon_h^z\|_{\mathcal T_h} + h^{-\frac 1 2} \|\varepsilon_h^z - \varepsilon_h^{\widehat z}\|_{\partial\mathcal T_h}).
	\end{align}
\end{lemma}

\begin{lemma}\label{lemma:step5_main_lemma}
	We have
	\begin{subequations}
		\begin{align}
		\hspace{3em}&\hspace{-3em}\norm{\varepsilon_h^{\bm{p}}}_{\mathcal{T}_h}+h^{-\frac 1 2}\|{\varepsilon_h^z-\varepsilon_h^{\widehat{z}}}\|_{\partial \mathcal T_h}\nonumber \\ 
		&\lesssim  h^{s_{\bm p}}\norm{\bm p}_{s^{\bm p},\Omega} + h^{s_{z}-1}\norm{z}_{s^{z},\Omega} +h^{s_{\bm q}+1}\norm{\bm q}_{s^{\bm q},\Omega} + h^{s_{y}}\norm{y}_{s^{y},\Omega},\label{error_p}\\
		\norm{\varepsilon_h^{{z}}}_{\mathcal{T}_h} &\lesssim  h^{s_{\bm p}}\norm{\bm p}_{s^{\bm p},\Omega} + h^{s_{z}-1}\norm{z}_{s^{z},\Omega} +h^{s_{\bm q}+1}\norm{\bm q}_{s^{\bm q},\Omega} + h^{s_{y}}\norm{y}_{s^{y},\Omega}.\label{error_z}
		\end{align}
	\end{subequations}
\end{lemma}

\begin{proof}
	First, we note the key inequality in \Cref{nabla_ine} is valid with $ (z,\bm p, \hat z) $ in place of $ (y,\bm q, \hat y) $.  This gives
	\begin{align}\label{nabla_z}
	\|\nabla \varepsilon_h^z\|_{\mathcal T_h}  \le  \|\varepsilon^{\bm p}_h\|_{\mathcal T_h}+Ch^{-\frac1 2}\|\varepsilon^z_h-\varepsilon^{\widehat z}_h\|_{\partial\mathcal T_h},
	\end{align}
	which we use below.  Next, since $\varepsilon_h^{\widehat z}=0$ on $\varepsilon_h^\partial$, the basic property of $ \mathscr B_2 $ in \Cref{property_B} gives
	\[ \mathscr B_2 (\varepsilon_h^{\bm p},\varepsilon_h^{ z}, \varepsilon_h^{\widehat z}, \varepsilon_h^{\bm p},\varepsilon_h^{z}, \varepsilon_h^{\widehat z}) =(\varepsilon_h^{\bm{p}},\varepsilon_h^{\bm{p}})_{\mathcal{T}_h}+ \|(h^{-1}+\tau_2+\frac 12 \bm{\beta} \cdot\bm n)^{\frac 12 } (\varepsilon_h^z-\varepsilon_h^{\widehat{z}})\|_{\partial\mathcal T_h}^2. \]
	Then taking $(\bm r_2, w_2,\mu_2) = (\bm \varepsilon_h^{\bm p},\varepsilon_h^z,\varepsilon_h^{\widehat z})$ in \eqref{error_equation_L2k2} in \Cref{lemma:step4_first_lemma} gives
	\begin{align*}
	(\varepsilon_h^{\bm{p}},\varepsilon_h^{\bm{p}})_{\mathcal{T}_h}& + \|(h^{-1}+\tau_2+\frac 12 \bm{\beta} \cdot\bm n)^{\frac 12 } (\varepsilon_h^z-\varepsilon_h^{\widehat{z}})\|_{\partial\mathcal T_h}^2\\
	&=( \bm \beta \delta^z, \nabla \varepsilon_h^z)_{{\mathcal{T}_h}}  -\langle \widehat {\bm\delta}_2,\varepsilon_h^z - \varepsilon_h^{\widehat z}\rangle_{\partial\mathcal T_h} + (y-y_h(u),\varepsilon_h^z)_{\mathcal T_h}\\
	& =: T_1+T_2+T_3.
	\end{align*}
	By the same argument as in \Cref{lemma:step2_main_lemma},  simply applying \eqref{nabla_z} and Young's inequality gives
	\begin{align*}
	T_1 &=  ( \bm \beta \delta^z, \nabla \varepsilon_h^z)_{{\mathcal{T}_h}} \le C \| \delta^z\|_{\mathcal T_h}^2 + \frac 1 4
	\|\varepsilon_h^{\bm{p}}\|_{\mathcal T_h}^2 + \frac 1 {4h} \|{\varepsilon_h^z-\varepsilon_h^{\widehat{z}}}\|_{\partial \mathcal T_h}^2,\\
	T_2 &=  -\langle \widehat {\bm\delta}_2,\varepsilon_h^z - \varepsilon_h^{\widehat z}\rangle_{\partial\mathcal T_h} \le 4h  \|\widehat {\bm\delta}\|_{\partial\mathcal T_h}^2 + \frac 1 {4h} \|{\varepsilon_h^z-\varepsilon_h^{\widehat{z}}}\|_{\partial \mathcal T_h}^2.
	\end{align*}
	Finally, for the term $T_3$, we have
	\begin{align*}
	T_3&= (y-y_h(u),\varepsilon_h^z)_{\mathcal T_h} \le  \|y-y_h(u)\|_{\mathcal T_h} \|\varepsilon_h^z\|_{\mathcal T_h}\\
	&\le C\|y-y_h(u)\|_{\mathcal T_h} (\|\nabla \varepsilon_h^z\|_{\mathcal T_h} + h^{-\frac 1 2} \|\varepsilon_h^z - \varepsilon_h^{\widehat z}\|_{\partial\mathcal T_h})\\
	&\le  C\|y-y_h(u)\|_{\mathcal T_h} (\|\varepsilon^{\bm p}_h\|_{\mathcal T_h}+h^{-\frac1 2}\|\varepsilon^z_h-\varepsilon^{\widehat z}_h\|_{\partial\mathcal T_h})\\
	&\le  C\|y-y_h(u)\|_{\mathcal T_h}^2 + \frac 1 4
	\|\varepsilon_h^{\bm{p}}\|_{\mathcal T_h}^2 + \frac 1 {4h} \|{\varepsilon_h^z-\varepsilon_h^{\widehat{z}}}\|_{\partial \mathcal T_h}^2.
	\end{align*}
	Summing $T_1$ to $T_3$ gives
	\begin{align*}
	\hspace{3em}&\hspace{-3em} \norm{\varepsilon_h^{\bm{p}}}_{\mathcal{T}_h}+h^{-\frac 1 2}\|{\varepsilon_h^z-\varepsilon_h^{\widehat{z}}}\|_{\partial \mathcal T_h}\\
	&\lesssim  h^{s_{\bm p}}\norm{\bm p}_{s^{\bm p},\Omega} + h^{s_{z}-1}\norm{z}_{s^{z},\Omega} +h^{s_{\bm q}+1}\norm{\bm q}_{s^{\bm q},\Omega} + h^{s_{y}}\norm{y}_{s^{y},\Omega}.
	\end{align*}
	Finally, \eqref{poin_in}, \eqref{error_p}, and \eqref{nabla_z} together imply \eqref{error_z}.
\end{proof}

As a consequence, a simple application of the triangle inequality gives optimal convergence rates for $\|\bm p -\bm p_h(u)\|_{\mathcal T_h}$ and $\|z -z_h(u)\|_{\mathcal T_h}$:

\begin{lemma}\label{lemma:step6_conv_rates}
	We have
	\begin{subequations}
		\begin{align}
		\|\bm p -\bm p_h(u)\|_{\mathcal T_h}
		&\lesssim h^{s_{\bm p}}\norm{\bm p}_{s^{\bm p},\Omega}  + h^{s_{z}-1}\norm{z}_{s^{z},\Omega} +h^{s_{\bm q}+1}\norm{\bm q}_{s^{\bm q},\Omega} + h^{s_{y}}\norm{y}_{s^{y},\Omega},\\
		\|z -z_h(u)\|_{\mathcal T_h} &\lesssim  h^{s_{\bm p}}\norm{\bm p}_{s^{\bm p},\Omega} + h^{s_{z}-1}\norm{z}_{s^{z},\Omega} + h^{s_{\bm q}+1}\norm{\bm q}_{s^{\bm q},\Omega} + h^{s_{y}}\norm{y}_{s^{y},\Omega}.
		\end{align}
	\end{subequations}
\end{lemma}

\subsubsection{Step 6: Estimate for $\|u-u_h\|_{\varepsilon_h^\partial}$ and $\norm {y-y_h}_{\mathcal T_h}$.}

Next, we bound the error between the solutions of the auxiliary problem and the HDG problem \eqref{HDG_full_discrete}.  We use these error bounds and the error bounds in \Cref{lemma:step3_conv_rates}, \Cref{lemma:step5_main_lemma}, and \Cref{lemma:step6_conv_rates} to obtain the main results.

For the remaining steps, we denote
\begin{equation*}
\begin{split}
\zeta_{\bm q} &=\bm q_h(u)-\bm q_h,\quad\zeta_{y} = y_h(u)-y_h,\quad\zeta_{\widehat y} = \widehat y_h(u)-\widehat y_h,\\
\zeta_{\bm p} &=\bm p_h(u)-\bm p_h,\quad\zeta_{z} = z_h(u)-z_h,\quad\zeta_{\widehat z} = \widehat z_h(u)-\widehat z_h.
\end{split}
\end{equation*}
Subtracting the auxiliary problem and the HDG problem gives the following error equations
\begin{subequations}\label{eq_yh}
	\begin{align}
	\mathscr B_1(\zeta_{\bm q},\zeta_y,\zeta_{\widehat y};\bm r_1, w_1,\mu_1)&=-\langle P_M u-u_h, \bm r_1\cdot \bm n + (\bm{\beta} \cdot\bm n -h^{-1}-\tau_1)w_1\rangle_{\varepsilon_h^\partial},\label{eq_yh_yhu}\\
	\mathscr B_2(\zeta_{\bm p},\zeta_z,\zeta_{\widehat z};\bm r_2, w_2,\mu_2)&=(\zeta_y, w_2)_{\mathcal T_h}\label{eq_zh_zhu}.
	\end{align}
\end{subequations}

\begin{lemma}
	If \textbf{(A1)} and \textbf{(A2)} hold, then
	\begin{align*}
	 \gamma\norm{u-u_h}_{\varepsilon_h^\partial}^2  + \norm {\zeta_y}_{\mathcal T_h}^2 &= \langle \gamma u+\bm p_h(u)\cdot\bm n +h^{-1}  z_h(u) + \tau_2 z_h(u),u-u_h\rangle_{\varepsilon_h^\partial} \\
	&\quad- \langle \gamma u_h+\bm p_h\cdot\bm n + h^{-1}  z_h+\tau_2 z_h,u-u_h\rangle_{\varepsilon_h^\partial}.
	\end{align*}
\end{lemma}
\begin{proof}
	First, we have
	\begin{multline*}
	\langle \gamma u+\bm p_h(u)\cdot\bm n +h^{-1}  z_h(u) + \tau_2 z_h(u),u-u_h\rangle_{\varepsilon_h^\partial} - \langle \gamma u_h+\bm p_h\cdot\bm n + h^{-1}  z_h+\tau_2 z_h,u-u_h\rangle_{\varepsilon_h^\partial}\\
	= \gamma\norm{u-u_h}_{\varepsilon_h^\partial}^2 +\langle \zeta_{\bm p}\cdot\bm n+h^{-1}  \zeta_z +\tau_2 \zeta_z,u-u_h\rangle_{\varepsilon_h^\partial}.
	\end{multline*}
	Next, \Cref{identical_equa} gives
	\begin{align*}
	\mathscr B_1 &(\zeta_{\bm q},\zeta_y,\zeta_{\widehat{y}};\zeta_{\bm p},-\zeta_{z},-\zeta_{\widehat z}) + \mathscr B_2(\zeta_{\bm p},\zeta_z,\zeta_{\widehat z};-\zeta_{\bm q},\zeta_y,\zeta_{\widehat{y}})  = 0.
	\end{align*}
	On the other hand, since $\tau_2$ is piecewise constant on $\partial \mathcal T_h$, we have
	\begin{align*}
	\mathscr B_1(\zeta_{\bm q},\zeta_y,\zeta_{\widehat y};&\zeta_{\bm p}, -\zeta_{z},-\zeta_{\widehat{ z}}) + \mathscr B_2(\zeta_{\bm p},\zeta_z,\zeta_{\widehat z}; -\zeta_{\bm q}, \zeta_{y},\zeta_{\widehat{y}})\\
	&=  (\zeta_{ y},\zeta_{ y})_{\mathcal{T}_h} - \langle P_Mu-u_h, \zeta_{\bm p}\cdot \bm{n} + (h^{-1} +\tau_1-\bm{\beta} \cdot\bm n)\zeta_z \rangle_{{\varepsilon_h^{\partial}}}\\
	&= (\zeta_{ y},\zeta_{ y})_{\mathcal{T}_h} -\langle P_Mu-u_h, \zeta_{\bm p}\cdot \bm{n} + h^{-1}\zeta_z +\tau_2 \zeta_z  \rangle_{{\varepsilon_h^{\partial}}}\\
	&= (\zeta_{ y},\zeta_{ y})_{\mathcal{T}_h} -\langle u-u_h, \zeta_{\bm p}\cdot \bm{n} + h^{-1}\zeta_z +\tau_2 \zeta_z  \rangle_{{\varepsilon_h^{\partial}}}.
	\end{align*}
	Comparing the above two equalities gives
	\begin{align*}
	(\zeta_{ y},\zeta_{ y})_{\mathcal{T}_h} = \langle u-u_h, \zeta_{\bm p}\cdot \bm{n} + h^{-1}\zeta_z +\tau_2 \zeta_z  \rangle_{{\varepsilon_h^{\partial}}}.
	\end{align*}
\end{proof}

\begin{theorem}
	We have
	\begin{align*}
	\norm{u-u_h}_{\varepsilon_h^\partial}&\lesssim h^{s_{\bm p}-\frac 1 2}\norm{\bm p}_{s_{\bm p},\Omega} +  h^{s_{z}-\frac 3 2}\norm{z}_{s_{z},\Omega} + h^{s_{\bm q}+\frac 1 2}\norm{\bm q}_{s_{\bm q},\Omega} + h^{s_{y}-\frac 12 }\norm{y}_{s_{y},\Omega},\\
	\norm {y-y_h}_{\mathcal T_h} &\lesssim h^{s_{\bm p}-\frac 1 2}\norm{\bm p}_{s_{\bm p},\Omega} +  h^{s_{z}-\frac 3 2}\norm{z}_{s_{z},\Omega} + h^{s_{\bm q}+\frac 1 2}\norm{\bm q}_{s_{\bm q},\Omega} + h^{s_{y}-\frac 12 }\norm{y}_{s_{y},\Omega}.
	\end{align*}
\end{theorem}

\begin{proof}
	%
	Since $\gamma u+\bm p \cdot\bm n=0$ on $\varepsilon_h^{\partial}$ and $\gamma u_h+\bm p_h\cdot\bm n + h^{-1} z_h +\tau_2 z_h=0$ on $\varepsilon_h^{\partial}$ we have
	\begin{align*}
	\gamma\norm{u-u_h}_{\varepsilon_h^\partial}^2  + \norm {\zeta_y}_{\mathcal T_h}^2 &= \langle \gamma u+\bm p_h(u)\cdot\bm n + h^{-1} z_h(u) + \tau_2 z_h(u),u-u_h\rangle_{\varepsilon_h^\partial}\\
	&=\langle (\bm p_h(u)-\bm p)\cdot\bm n + h^{-1} z_h(u)+ \tau_2 z_h(u) ,u-u_h\rangle_{\varepsilon_h^\partial}.
	\end{align*}
	Next, since $\widehat z_h(u) = z=0$ on $\varepsilon_h^{\partial}$ we have
	\begin{align*}
	\norm {\bm p_h(u)-\bm p}_{\partial \mathcal T_h} &\le \norm {\bm p_h(u)-\bm{\Pi}\bm p}_{\partial \mathcal T_h} +\norm {\bm{\Pi}\bm p - \bm p}_{\partial \mathcal T_h}\\
	&\lesssim h^{-\frac 1 2}\norm {\varepsilon_h^{\bm p}}_{\mathcal T_h} +h^{s_{\bm p}-\frac 12 } \norm{\bm p}_{s^{\bm p},\Omega},\\
	\|z_h(u)\|_{\varepsilon_h^\partial} &=\|z_h(u) -\Pi z +P_M z-\widehat z_h(u) \|_{\varepsilon_h^\partial} =  \|\varepsilon_h^z -\varepsilon_h^{\widehat z}\|_{\partial\mathcal T_h}.
	\end{align*}
	Some simple manipulations  gives
	\begin{align*}
	\norm{u-u_h}_{\varepsilon_h^\partial}  + \|\zeta_y\|_{\mathcal T_h}&\lesssim h^{-\frac 1 2}\norm {\varepsilon_h^{\bm p}}_{\mathcal T_h} +h^{s_{\bm p}-\frac 12 } \norm{\bm p}_{s^{\bm p},\Omega} +h^{-1}\|\varepsilon_h^z -\varepsilon_h^{\widehat z}\|_{\partial\mathcal T_h}.
	\end{align*}
	By \Cref{lemma:step5_main_lemma} and properties of the $ L^2 $ projection, we have
	\begin{align*}
	\hspace{3em}&\hspace{-3em} \norm{u-u_h}_{\varepsilon_h^\partial}  + \|\zeta_y\|_{\mathcal T_h}\\
	&\lesssim h^{s_{\bm p}-\frac 1 2}\norm{\bm p}_{s^{\bm p},\Omega} + h^{s_{z}-\frac 3 2}\norm{z}_{s^{z},\Omega} +h^{s_{\bm q}+\frac 1 2}\norm{\bm q}_{s^{\bm q},\Omega} + h^{s_{y}-\frac 1 2}\norm{y}_{s^{y},\Omega}.
	\end{align*}
	Then, by the triangle inequality and \Cref{lemma:step3_conv_rates} we obtain
	\begin{align*}
	\|y -y_h\|_{\mathcal T_h}\lesssim h^{s_{\bm p}-\frac 1 2}\norm{\bm p}_{s^{\bm p},\Omega} + h^{s_{z}-\frac 3 2}\norm{z}_{s^{z},\Omega} +h^{s_{\bm q}+\frac 1 2}\norm{\bm q}_{s^{\bm q},\Omega} + h^{s_{y}-\frac 1 2}\norm{y}_{s^{y},\Omega}.
	\end{align*}
\end{proof}

\subsubsection{Step 7: Estimates for $\|\boldmath p-\boldmath p_h\|_{\mathcal T_h}$,  $\|z-z_h\|_{\mathcal T_h}$ and $\|\boldmath q - \boldmath q_h\|_{\mathcal T_h}$ .}
\begin{lemma}
	We have
	\begin{align*}
	\norm {\zeta_{\bm p}}_{\mathcal T_h}  &\lesssim h^{s_{\bm p}-\frac 1 2}\norm{\bm p}_{s_{\bm p},\Omega} +  h^{s_{z}-\frac 3 2}\norm{z}_{s_{z},\Omega} + h^{s_{\bm q}+\frac 1 2}\norm{\bm q}_{s_{\bm q},\Omega} + h^{s_{y}-\frac 12 }\norm{y}_{s_{y},\Omega},\\
	\|\zeta_z\|_{\mathcal T_h} & \lesssim  h^{s_{\bm p}-\frac 1 2}\norm{\bm p}_{s_{\bm p},\Omega} +  h^{s_{z}-\frac 3 2}\norm{z}_{s_{z},\Omega} + h^{s_{\bm q}+\frac 1 2}\norm{\bm q}_{s_{\bm q},\Omega} + h^{s_{y}-\frac 12 }\norm{y}_{s_{y},\Omega}.
	\end{align*}
\end{lemma}
\begin{proof}
	By \Cref{property_B}, the error equation \eqref{eq_zh_zhu}, and since $\zeta_{ \widehat z} = 0$ on $\varepsilon_h^{\partial}$, we have
	\begin{align*}
	\hspace{2em}&\hspace{-2em}\mathscr B_2(\zeta_{\bm p},\zeta_z,\zeta_{\widehat z};\zeta_{\bm p},\zeta_z,\zeta_{\widehat z})\\
	&=(\zeta_{\bm p}, \zeta_{\bm p})_{{\mathcal{T}_h}}+\langle (h^{-1}+\tau_2+\frac 12 \bm{\beta}\cdot\bm n) (\zeta_z-\zeta_{\widehat z}) , \zeta_z-\zeta_{\widehat z}\rangle_{\partial{{\mathcal{T}_h}}}\\
	&=(\zeta_y,\zeta_z)_{\mathcal T_h}\\
	&\le \norm{\zeta_y}_{\mathcal T_h} \norm{\zeta_z}_{\mathcal T_h}\\
	&\lesssim \norm{\zeta_y}_{\mathcal T_h} (\|\nabla \zeta_z\|_{\mathcal T_h} + h^{-\frac 1 2} \|\zeta_z - \zeta_{\widehat z}\|_{\partial\mathcal T_h}) \\
	& \lesssim \norm{\zeta_y}_{\mathcal T_h}
	(\|\zeta_{\bm p}\|_{\mathcal T_h}+h^{-\frac1 2}\|\zeta_z-\zeta_{\widehat z}\|_{\partial\mathcal T_h}),
	\end{align*}
	where we used the discrete Poincar{\'e} inequality in \Cref{lemma:discr_Poincare_ineq} and also \Cref{nabla_ine}.  This implies
	\begin{align*}
	\hspace{3em}&\hspace{-3em} \norm {\zeta_{\bm p}}_{\mathcal T_h} +h^{-\frac1 2}\|\zeta_z-\zeta_{\widehat z}\|_{\partial\mathcal T_h}\\
	& \lesssim h^{s_{\bm p}-\frac 1 2}\norm{\bm p}_{s_{\bm p},\Omega} +  h^{s_{z}-\frac 3 2}\norm{z}_{s_{z},\Omega} + h^{s_{\bm q}+\frac 1 2}\norm{\bm q}_{s_{\bm q},\Omega} + h^{s_{y}-\frac 12 }\norm{y}_{s_{y},\Omega}.
	\end{align*}
	
	The discrete Poincar{\'e} inequality in \Cref{lemma:discr_Poincare_ineq} also gives
	\begin{align*}
	\|\zeta_z\|_{\mathcal T_h} & \lesssim \|\nabla \zeta_z\|_{\mathcal T_h} + h^{-\frac 1 2} \|\zeta_z - \zeta_{\widehat z}\|_{\partial\mathcal T_h}\\
	&\lesssim h^{s_{\bm p}-\frac 1 2}\norm{\bm p}_{s_{\bm p},\Omega} +  h^{s_{z}-\frac 3 2}\norm{z}_{s_{z},\Omega} + h^{s_{\bm q}+\frac 1 2}\norm{\bm q}_{s_{\bm q},\Omega} + h^{s_{y}-\frac 12 }\norm{y}_{s_{y},\Omega}.
	\end{align*}
\end{proof}

\begin{lemma}
	If \textbf{(A1)}  and $k\geq 1$ hold, then
	\begin{align*}
	\norm {\zeta_{\bm q}}_{\mathcal T_h}  &\lesssim h^{s_{\bm p}-1}\norm{\bm p}_{s_{\bm p},\Omega} +  h^{s_{z}-2}\norm{z}_{s_{z},\Omega} + h^{s_{\bm q}}\norm{\bm q}_{s_{\bm q},\Omega} + h^{s_{y}-1 }\norm{y}_{s_{y},\Omega}.
	\end{align*}
\end{lemma}
\begin{proof}
	By \Cref{property_B}, the error equation \eqref{eq_yh_yhu}, and since $ \tau_2 $ is piecewise constant on $ \partial \mathcal{T}_h $, we have
	\begin{align*}
	\hspace{2em}&\hspace{-2em}\mathscr B_1(\zeta_{\bm q},\zeta_y,\zeta_{\widehat y};\zeta_{\bm q},\zeta_y,\zeta_{\widehat y}) \\
	&=(\zeta_{\bm q}, \zeta_{\bm q})_{{\mathcal{T}_h}}+\langle (h^{-1}+\tau_1 - \frac 12 \bm \beta \cdot\bm n) (\zeta_y-\zeta_{\widehat y}) , \zeta_y-\zeta_{\widehat y}\rangle_{\partial{{\mathcal{T}_h}}\backslash\varepsilon_h^\partial} - (\nabla\cdot\bm{\beta} \zeta_y,\zeta_y)_{\mathcal T_h}\\
	&\quad+ \langle (h^{-1}+\tau_1 - \frac 12 \bm \beta \cdot\bm n) \zeta_y, \zeta_y \rangle_{\varepsilon_h^\partial} \\
	&= -\langle P_M u-u_h, \zeta_{\bm q}\cdot \bm{n} + (\bm{\beta}\cdot\bm n-h^{-1}-\tau_1) \zeta_y \rangle_{{\varepsilon_h^{\partial}}}\\
	&= -\langle P_Mu-u_h, \zeta_{\bm q}\cdot \bm{n} - (h^{-1}+\tau_2) \zeta_y \rangle_{{\varepsilon_h^{\partial}}}\\
	&=-\langle u-u_h, \zeta_{\bm q}\cdot \bm{n} - (h^{-1}+\tau_2) \zeta_y \rangle_{{\varepsilon_h^{\partial}}}\\
	&	\lesssim \norm {u-u_h}_{\varepsilon_h^{\partial}} (\norm {\zeta_{\bm q}}_{\varepsilon_h^{\partial}} + h^{-1} \norm {\zeta_{y}}_{\varepsilon_h^{\partial}} )\\
	&	\lesssim h^{-\frac 1 2}\norm {u-u_h}_{\varepsilon_h^{\partial}} (\norm {\zeta_{\bm q}}_{\mathcal T_h} + h^{-\frac 1 2} \norm {\zeta_{y}}_{\varepsilon_h^{\partial}}),
	\end{align*}
	which gives
	\begin{align*}
	\norm {\zeta_{\bm q}}_{\mathcal T_h} &\lesssim h^{-\frac 1 2}\norm {u-u_h}_{\varepsilon_h^{\partial}} \\
	&\lesssim h^{s_{\bm p}-1}\norm{\bm p}_{s_{\bm p},\Omega} +  h^{s_{z}-2}\norm{z}_{s_{z},\Omega} + h^{s_{\bm q}}\norm{\bm q}_{s_{\bm q},\Omega} + h^{s_{y}-1 }\norm{y}_{s_{y},\Omega}.
	\end{align*}
\end{proof}

The above lemma along with the triangle inequality, \Cref{lemma:step3_conv_rates}, and \Cref{lemma:step6_conv_rates} complete the proof of the main result:
\begin{thm}
	We have
	\begin{align*}
	\norm {\bm p - \bm p_h}_{\mathcal T_h}   &\lesssim h^{s_{\bm p}-\frac 1 2}\norm{\bm p}_{s_{\bm p},\Omega} +  h^{s_{z}-\frac 3 2}\norm{z}_{s_{z},\Omega} + h^{s_{\bm q}+\frac 1 2}\norm{\bm q}_{s_{\bm q},\Omega} + h^{s_{y}-\frac 12 }\norm{y}_{s_{y},\Omega},\\
	\norm {z - z_h}_{\mathcal T_h}   & \lesssim  h^{s_{\bm p}-\frac 1 2}\norm{\bm p}_{s_{\bm p},\Omega} +  h^{s_{z}-\frac 3 2}\norm{z}_{s_{z},\Omega} + h^{s_{\bm q}+\frac 1 2}\norm{\bm q}_{s_{\bm q},\Omega} + h^{s_{y}-\frac 12 }\norm{y}_{s_{y},\Omega}.
	\end{align*}
	If in addition $ k \geq 1 $, then
	\begin{align*}
	\norm {\bm q - \bm q_h}_{\mathcal T_h} &\lesssim h^{s_{\bm p}-1}\norm{\bm p}_{s_{\bm p},\Omega} +  h^{s_{z}-2}\norm{z}_{s_{z},\Omega} + h^{s_{\bm q}}\norm{\bm q}_{s_{\bm q},\Omega} + h^{s_{y}-1 }\norm{y}_{s_{y},\Omega}.
	\end{align*}
\end{thm}

  \section{Numerical Experiments}
\label{sec:numerics}

We present numerical results for a 2D example problem on a square domain $\Omega = [0,1/8]\times [0,1/8] $.  For the results presented below, we chose $ \tau_1 = \tau_2 = 1 $ for the stabilization functions.  In other numerical experiments not reported here, we also chose $ \tau_1 $ and $ \tau_2 $ to satisfy the conditions \textbf{(A1)}-\textbf{(A3)} and we obtained similar results.

The problem data is chosen as
\begin{align*}
f=0, \ \  y_d = (x^2+y^2)^{s},\ \  \bm \beta = [1, 1], \ \ \ \mbox{and} \  \ \ \gamma = 1,
\end{align*}
where $s=-10^{-5}$.  Therefore $ y_d $ has a singularity, but $y_d\in H^{1-\varepsilon}(\Omega) $ for all $\varepsilon> 2\times 10^{-5}$.  Since the largest interior angle is $ \omega = {\pi}/{2}$, we have $ r_\Omega = 3/2-\varepsilon $ for all $\varepsilon> 2\times 10^{-5}$. 

An exact solution for this problem is not known, and therefore we compare the approximate solutions computed using various values of $ h $ and a reference solution computed on a fine mesh with 524288 elements and $h = 2^{-12}\sqrt 2$.  

For $k=1$, \Cref{cor:main_result} in \Cref{sec:analysis}  gives the convergence rates
\begin{align*}
&\norm{y-{y}_h}_{0,\Omega}=O( h^{3/2-\varepsilon} ),\qquad  \;\norm{z-{z}_h}_{0,\Omega}=O( h^{3/2-\varepsilon} ),\\
&\norm{\bm{q}-\bm{q}_h}_{0,\Omega} = O( h^{1-\varepsilon} ),\quad  \quad\;\; \norm{\bm{p}-\bm{p}_h}_{0,\Omega} = O( h^{3/2-\varepsilon} ),
\end{align*}
and
\begin{align*}
&\norm{u-{u}_h}_{0,\Gamma} = O( h^{3/2-\varepsilon}).
\end{align*}
\Cref{table_1} shows the computed errors for a variety of mesh sizes.  The convergence rates for the optimal control $ u $ and the flux $ \bm q $ are precisely predicted by the convergence theory presented here.  The convergence rates for the other variables are higher than predicted by the theory; this phenomenon has been observed numerically in other works on numerical methods for boundary control problems \cite{HuShenSinglerZhangZheng_HDG_Dirichlet_control1,MR3070527,MR3317816,MR2806572}.
\begin{table}
	\begin{center}
		\begin{tabular}{|c|c|c|c|c|c|}
			\hline
			$h/\sqrt{2}$ &$2^{-4}$& $2^{-5}$&$2^{-6}$&$2^{-7}$ & $2^{-8}$ \\
			\hline
			$\norm{\bm{q}-\bm{q}_h}_{0,\Omega}$&2.57e-2& 1.32e-2&6.66e-3  & 3.35e-3 & 1.68e-3 \\
			\hline
			order&-& 0.96& 0.98  &1.00& 1.00 \\
			\hline
			$\norm{\bm{p}-\bm{p}_h}_{0,\Omega}$&5.01e-4 & 1.57e-4&4.57e-5 &1.29e-5 &3.55e-6\\
			\hline
			order&-&  1.68&1.78 &1.83 & 1.86 \\
			\hline
			$\norm{{y}-{y}_h}_{0,\Omega}$&2.00e-4& 5.04e-5 &1.29e-5 & 3.26e-6  & 8.21e-7\\
			\hline
			order&-& 1.99&1.99&1.98 & 1.99\\
			\hline
			$\norm{{z}-{z}_h}_{0,\Omega}$& 3.41e-6& 5.20e-7&7.60e-8  &1.08e-8 & 1.49e-9 \\
			\hline
			order&-& 2.71&2.77&2.82& 2.85 \\
			\hline
			$\norm{{u}-{u}_h}_{0,\Gamma}$&2.40e-3& 9.04e-4&3.27e-4 &1.17e-4  & 4.18e-5\\
			\hline
			order&-&  1.41& 1.47&1.48& 1.49 \\
			\hline
		\end{tabular}
	\end{center}
	\caption{2D Example with $k=1$: Errors for the control $u$, state $y$, adjoint state $z$, and the fluxes $\bm q$ and $\bm p$.}\label{table_1}
\end{table}

For the case $k=0$, \Cref{cor:main_result} in \Cref{sec:analysis}  gives the convergence rates
\begin{align*}
\norm{y-{y}_h}_{0,\Omega}=O( h^{1/2} ),~~  \;\norm{z-{z}_h}_{0,\Omega}=O( h^{1/2} ),~~ \norm{\bm{p}-\bm{p}_h}_{0,\Omega} = O( h^{1/2} ),
\end{align*}
and
\begin{align*}
&\norm{u-{u}_h}_{0,\Gamma} = O( h^{1/2}).
\end{align*}
As mentioned in \Cref{sec:analysis}, the convergence rates for $ k = 0 $ obtained in \Cref{cor:main_result} are suboptimal.  Numerical results are reported in \Cref{table_2}, and all convergence rates are higher than predicted by the theory.  Obtaining optimal convergence rates for the $ k = 0 $ case is an interesting topic for future work.
\begin{table}
	\begin{center}
		\begin{tabular}{|c|c|c|c|c|c|}
			\hline
			$h/\sqrt{2}$ &$2^{-4}$& $2^{-5}$&$2^{-6}$&$2^{-7}$ & $2^{-8}$\\
			\hline
			$\norm{\bm{q}-\bm{q}_h}_{0,\Omega}$&4.67e-2& 3.27e-2&2.01e-2   & 1.18e-2 & 6.66e-3 \\
			\hline
			order&-&  0.51   &0.70   &0.76   &0.82 \\
			\hline
			$\norm{\bm{p}-\bm{p}_h}_{0,\Omega}$&1.61e-3  & 9.22e-4&4.81e-4 &2.44e-4 &1.22e-4\\
			\hline
			order&-&  0.80   &0.94   &0.98   &1.00 \\
			\hline
			$\norm{{y}-{y}_h}_{0,\Omega}$&6.54e-4& 2.77e-4 &9.10e-5  & 2.83e-5  & 8.80e-6 \\
			\hline
			order&-& 1.24   &1.60   &1.68   &1.69 \\
			\hline
			$\norm{{z}-{z}_h}_{0,\Omega}$& 6.54e-5& 1.82e-5&4.74e-6  &1.20e-6 & 3.02e-7 \\
			\hline
			order&-& 1.85   &1.94   &2.00   &2.00\\
			\hline
			$\norm{{u}-{u}_h}_{0,\Gamma}$&5.88e-3 & 3.50e-3&1.92e-3   &1.01e-3 & 5.21e-4 \\
			\hline
			order&-&  0.75   &0.87   &0.92   &0.96 \\
			\hline
		\end{tabular}
	\end{center}
	\caption{2D Example with $k=0$: Errors for the control $u$, state $y$, adjoint state $z$, and the fluxes $\bm q$ and $\bm p$.}\label{table_2}
\end{table}

\section{Conclusion}

We considered a Dirichlet boundary control problem for an elliptic convection diffusion equation and made two contributions.  First, for a polygonal domain we considered a very weak mixed formulation of the PDE, and established well-posedness and regularity results for the PDE and the optimal control problem.  Next, we proposed a new HDG method to approximate the solution of the optimality system and established optimal superlinear convergence rates for the control under certain assumptions on the domain and the desired state.  We presented numerical results to demonstrate the performance of the method.

In the second part of this work \cite{HuMateosSinglerZhangZhang17}, we remove the restrictions on the domain and the desired state and use very different analysis techniques to prove optimal convergence rates for the control.

As far as we are aware, this is the first work to explore the analysis of this Dirichlet control problem and the numerical analysis of a computational method for this problem.  There are a number of topics that can be explored further, including an improved numerical analysis for the $ k = 0 $ case, the convergence analysis of other HDG methods for this problem, and HDG methods for more challenging Dirichlet control problems.

\section*{Acknowledgments} The authors thank Bernardo Cockburn for helpful conversations.

\bibliographystyle{amsplain}
\bibliography{yangwen_ref_papers,yangwen_ref_books}

\end{document}